\documentclass{article}
\usepackage{amsmath}
\usepackage{amsthm}
\usepackage{amsfonts}
\usepackage[left=2cm,right=2cm]{geometry}
\usepackage{url}
\usepackage{enumitem}
\setlist[description,1]{leftmargin=1.0em,labelindent=0em}
\renewcommand{\emph}[1]{{\it #1}}
\allowdisplaybreaks

\newtheorem{thm}{Theorem}
\newtheorem{lem}[thm]{Lemma}

\newtheorem{cor}[thm]{Corollary}
\newtheorem*{ack}{Acknowledgement}
\newcommand{\mR}{{\mathbb R}}
\newcommand{\mT}{{\mathbb T}}
\newcommand{\trsp}{{\mathsf{trsp}}}
\newcommand{\Deltainv}{\Delta^{\!\scriptstyle{-1}}}
\newcommand{\vu}{{\mathbf u}}
\newcommand{\vV}{{\mathbf V}}
\newcommand{\vValpnodiv}{\vV_\alpha^{\textnormal{\tiny no-div}}}

\newcommand{\vb}{{\mathbf b}}
\newcommand{\vB}{{\mathbf B}}
\newcommand{\vez}{{\mathbf e_z}}
\newcommand{\eqdef}{\mathrel{:=}}
\newcommand{\fr}[1]{\frac1{#1}}
\newcommand{\uh}{{\mathrm u}_{\mathsf h}}
\newcommand{\bh}{{\mathrm b}_{\mathsf h}}
\newcommand{\wh}{{\mathrm w}_{\mathsf h}}
\newcommand{\vw}{{\mathbf w}}
\newcommand{\uhlim}{{\bar{\mathrm u}}_{\mathsf h}}
\newcommand{\bhlim}{{\bar{\mathrm b}}_{\mathsf h}}
\newcommand{\rlim}{\bar{r}}
\newcommand{\uvlim}{{\bar{u}}_3}
\newcommand{\bvlim}{{\bar{b}}_3}
\newcommand{\vVlim}{\overline{\vV}}

\newcommand{\uvcor}{u_3^{\mathsf{(cor)}}}
\newcommand{\rcor}{r^{\mathsf{(cor)}}}
\newcommand{\ep}{\varepsilon}
\newcommand{\epA}{\ep_{\textnormal{\tiny A}}}
\newcommand{\epM}{\ep_{\textnormal{\tiny M}}}
\newcommand{\pa}{\partial}

\newcommand{\vA}{\mathsf{A}}
\newcommand{\grad}{\nabla}
\newcommand{\pt}{\pa_t}
\newcommand{\pz}{\pa_z}

\newcommand{\ddt}{\frac{d}{dt}}
\newcommand{\cLM}{\mathcal{L}_\textnormal{\tiny M}}
\newcommand{\cLA}{\mathcal{L}_\textnormal{\tiny A}}
\newcommand{\cPdh}{\mathcal{P}^\textnormal{div}_{\mathsf h}}
\newcommand{\eval}[1]{{\big|_{#1}}}
\DeclareMathOperator{\az}{{\mathfrak a}_{\mathrm z}}
\DeclareMathOperator{\avv}{{av\!}}
\newcommand{\nab}{\nabla}
\newcommand{\nabh}{\nabla_{\mathsf h}}
\newcommand{\nc}{\nabla {\cdot} }
\newcommand{\nch}{\nabh{\cdot}}
\newcommand{\cn}{ \!\cdot\! \nabla}
\newcommand{\cnh}{ \!\cdot\! \nabh}
\newcommand{\boldzero}{\mathbf{0}}
\newcommand{\Al}{Alfv\'en }
\begin{document}
\title{Convergence Rate Estimates for the Low Mach and Alfv\'en Number Three-Scale Singular Limit of Compressible Ideal
  Magnetohydrodynamics}
    \author{Bin Cheng\thanks{{Department of Mathematics, University of
          Surrey, Guildford, GU2 7XH, United Kingdom}}
      \and Qiangchang Ju\thanks{Institute of Applied Physics and Computational Mathematics, P.O. Box 8009, Beijing 100088, China}
      \and Steve Schochet\thanks{School of Mathematical Sciences, Tel-Aviv University, Tel Aviv 69978, Israel}
    }
\maketitle
\begin{abstract}  Convergence rate estimates are obtained for singular limits of the compressible ideal magnetohydrodynamics equations, in which the Mach and
  Alfv\'en numbers tend to zero at different rates. The proofs use a detailed analysis of exact and approximate fast, intermediate, and slow modes
  together with improved estimates for the solutions and their time derivatives, and the time-integration method.  When the small parameters are
  related by a power law the convergence rates are positive powers of the Mach number, with the power varying depending on the component and the norm.
  Exceptionally, the convergence rate for two components involve the ratio of the two parameters, and that rate is proven to be sharp via corrector
  terms. Moreover, the convergence rates for the case of a power-law relation between the small parameters tend to the two-scale convergence rate as
  the power tends to one. These results demonstrate that the issue of convergence rates for three-scale singular limits, which was not addressed in
  the authors' previous paper, is much more complicated than for the classical two-scale singular limits.
\end{abstract}
\begin{flushleft}
\end{flushleft}

\section{Introduction}

A uniform existence theorem and a convergence theorem as the small parameters tend to zero were recently developed \cite{CJS17} for singular limits of
symmetric hyperbolic systems of
the form
\begin{equation}
\label{eq:general}
A_0(\epM \vV)\,\vV_t+\sum_{i=1}^d A_i(\vV)\,\vV_{x_i}=\fr{\epA}\cLA\vV+\fr{\epM}\cLM \vV,
\end{equation}
where $\epA$ and $\epM$ are small positive parameters and $\cLA$ and $\cLM$ are skew-adjoint constant-coefficient first-order differential operators.
If $\frac{\epA}{\epM}$ tends to zero as the parameters tend to zero then
systems of the form \eqref{eq:general} have three time scales: $O(\fr\epA)$, $O(\fr\epM)$, and $O(1)$. 

In this paper we begin the study of the rate of convergence of solutions of three-scale singular limits to corresponding solutions of their limit
equations, an issue that was not considered in~\cite{CJS17},
but which is significant for applications because it determines the accuracy of using the limiting dynamics to approximate the original system.
The convergence rate in the general case will undoubtedly be very complicated, since in
general many different limit systems are obtained for different power-law relations between the two small parameters as they both tend to zero
\cite[\S 4]{CJS17}.  As a first step, we study here the particular case of the low Mach and
Alfv\'en number limit of the compressible ideal
magnetohydrodynamics (MHD) equations in the presence of a large uniform magnetic field, which was the motivating example for our work. As we will
show, that system has essentially only one limit system, although the limit
\begin{equation}
  \label{eq:mulimdef}
  \mu_{\lim}\eqdef\lim_{\epA,\epM\to0} \mu
\end{equation}
of the ratio
\begin{equation}\label{eq:mudef}
  \mu\eqdef\tfrac{\epA}{\epM}
\end{equation}
appears in that limit system as a parameter. Even for the MHD system the study of the convergence rate is much more intricate than for
two-scale singular limits 
because, as described below, the bounds on the first time derivative satisfied by solutions to three-scale
systems are weaker than those satisfied for two-scale systems, and the eigenspace projections of the large operator $\fr{\epA}\cLA+\fr{\epM}\cLM$
depend on the ratio of the small parameters instead of being fixed as in the two-scale case.

The MHD system in three spatial dimensions that we study, derived in \S\ref{sec:curl} from a standard formulation of that system, is
\begin{subequations}\label{MHDsy}
\begin{align}
  \label{MHDsy:r}a(\epM r)\,&\big(\pt r+(\vu\cn) r)\;+\tfrac{a(\epM r)\rho(\epM r)}{\epM}\nc\vu=0 
  \phantom{v\;}
  \\
\label{MHDsy:vu}\rho(\epM r)\,&\big(\pt\vu+(\vu\cn )\vu\big)
     +\tfrac{a(\epM r)\rho(\epM r)}{\epM} \grad r 
  + \nabla\tfrac{|\vb|^2}{2}\phantom{\vb}-(\vb\cn)\vb  =\tfrac{\pz\vb-\nabla b_3}{\epA}, \\
  \label{MHDsy:vb}&\phantom{\big(}\pt\vb+(\vu\cn)\vb\phantom{\big)+\tfrac{a(\epM r)\rho(\epM r)}{\epM}\grad r\,}
  +(\nc\vu)\vb-(\vb\cn)\vu  =\tfrac{\pz\vu- \vez\,\nc\vu}{\epA},\\
\label{eq:divb0}
&\phantom{\big(}\nc\vb=0,  
\end{align}\end{subequations}
where 
\begin{equation}
  \label{eq:rhodef2}
  a(s)\eqdef\frac{p'(1+s)}{1+s}, \qquad \rho(s)\eqdef 1+s. 
\end{equation}

The divergence-free condition \eqref{eq:divb0} on the magnetic field
is preserved by the dynamics of \eqref{MHDsy:vb}, and so is just a restriction on the initial data.
 Hence straightforward calculations show that the
 system \eqref{MHDsy} has the form \eqref{eq:general}, with $\vV=(r,\vu,\vb)$.

Our main result is a rate of convergence of solutions of the MHD system~\eqref{MHDsy} as the small parameters $\epA$ and $\epM$ tend to
  zero. As a preliminary, we prove a uniform existence and a convergence result, including determining the limit system.
  Before stating these results we discuss notations, operators, initial data, and parameters that will be used in the statement of the theorem.
  First, we let $\|\,\|_{k}$ denote the $H^k$ norm. 
Next, for any vector $\vw$ let its ``horizontal'' part $\wh$ denote its
$x,y$ components $\left(\begin{smallmatrix} w_1\\w_2\end{smallmatrix}\right)$, and let $\cPdh$ denote the {\it two-dimensional}
Leray-Helmholtz projection onto divergence-free velocity fields in the $x,y$
plane or $2$-torus, i.e.
\begin{equation}\label{def:cPdh0}
\cPdh\wh:= \wh-\nabh\Delta_{\mathsf h}^{-1}\nch\wh,
\qquad \text{
where $\nabh:=\left(\begin{smallmatrix}\pa_x\\\pa_y\end{smallmatrix}\right)$,\ \  $\Delta_{\mathsf h}:=  \pa_x^2+\pa_y^2$}.
\end{equation}
The large terms in \eqref{MHDsy} form the ``Alfv\'en'' and ``Mach'' operators
\begin{equation}\label{eq:largeops}
  \cLA\vV:=\left(
\begin{smallmatrix}
0\\\left(
  \begin{smallmatrix}
     -\nabh b_3+\pz \bh  \\ 0
  \end{smallmatrix}
\right)
\\
 \left(
\begin{smallmatrix}
  \pz \uh \\ -\nch \uh
\end{smallmatrix}
\right)
\end{smallmatrix}
\right)
,\qquad
  \cLM\vV:= \left(
\begin{smallmatrix}
  -\nc\vu\\\left(
  \begin{smallmatrix}
   -\nabh r \\ -\pz r 
  \end{smallmatrix}
\right)
\\
0_3
\end{smallmatrix}
\right)
,
\end{equation}
where for notational convenience we have normalized the pressure law $p(\rho)$ to satisfy satisfy
\begin{equation}\label{eq:pp1}
  p'(1)=1
\end{equation}
by rescaling $\epM$.
Also,  the full average $\avv f$ and the vertical average $\az\! f$ of any function~$f$ defined on the $3$-torus~$\mT^3$ are
  \begin{equation}\label{eq:avvdef}
    \avv f\eqdef \frac{\iiint f\,dx\,dy\,dz}{\iiint 1\,dx\,dy\,dz}, \qquad  (\az f)(x,y):= \frac{\int f(x,y,z)\,dz}{\int 1\,dz}.
  \end{equation}

  Although the uniform existence and convergence results require that the initial data
satisfy the ``well-preparedness'' condition
\begin{equation}\label{eq:wellprep}
    \|(\epA^{-1}\cLA+\epM^{-1}\cLM)\vV^0\|_{{n-1}}\le c,
  \end{equation}
they do not require any assumption about the rate at which the initial data converge to their
  limit. However, such an assumption is obviously required in order to obtain a rate of convergence of solutions of the PDE. In the convergence rate result we
  assume that the initial data has the form developed in \S\ref{sec:fis}, which is a specialization of the general form of initial data satisfying
  \eqref{eq:wellprep}. Specifically, after expanding the initial data in powers of the small parameters and their ratio, the leading-order terms are
  assumed to be independent of the small parameters in order to avoid degrading the convergence rate. However, valid estimates for any initial
  data satisfying \eqref{eq:wellprep} below can be
  obtained simply by adding the size of the difference of the initial data when that difference is larger than the estimates obtained below.

When the parameter $\mu$ in \eqref{eq:mudef} is fixed then
\eqref{eq:general} and \eqref{MHDsy} essentially contain
only one small parameter and hence have only two time scales.
Uniform existence and convergence results for initial-value problems
of general systems containing one small parameter
were obtained in \cite{KM81}.
Moreover, their results remain valid with only cosmetic changes to the
proofs whenever $\mu_{\lim}>0$. 
Convergence rate theorems for both specific and general two-scale systems have been proven in
  \cite{KM82,Sch:rate,Sch:limits,Cheng:SIMA:2012,Cheng:Mach}.
We therefore focus on the more challenging case when $\mu\to0$, although our results will be phrased so as to remain valid when $\mu_{\lim}>0$.
It will be convenient to express our convergence results using just powers of $\epM$, by
defining a parameter~$\nu$ determined by
 \begin{equation}
  \label{eq:nudef}
  \epA=\epM^{1+\nu}, \quad\text{or, equivalently,} \quad \mu=\epM^\nu,
  \quad \text{i.e.,} \quad \nu\eqdef  \tfrac{\ln(\fr\epA)}{\ln(\fr\epM)}-1=\tfrac{\ln(\fr \mu)}{\ln(\fr\epM)},
\end{equation}
where in view of the results for two-scale singular limits
we will assume for notational simplicity that
\begin{equation}
  \label{eq:epAltepM}
  \epA< \epM, \qquad\text{i.e., $\mu<1$ and $\nu>0$.}
\end{equation}

To simplify the exposition we will assume that the spatial domain is periodic. The uniform existence result remains valid with the same proof when the
spatial domain is $\mathbb{R}^3$, while the the limit system is then identically zero because it is independent of the vertical coordinate.
  
Throughout this paper $c$ and $C$ denote positive constants that are independent of $\epA$ and $\epM$, which may take different values in each appearance.

\begin{thm}\label{thm:theom1}
  Let $n\ge 3$ be an integer. Assume that the spatial domain is $\mT^3$
  and that the small parameters are restricted to the region
\begin{equation}
  \label{eq:epAepMgen}
  0<\epM\le \epM^0\qquad\text{and}\qquad \epA\ge c\,\epM^{1+\fr{n-1}}.
\end{equation}
Assume in addition that the initial data $\vV^0\eqdef(r^0, \vu^0, \vb^0)$ for
system \eqref{MHDsy},
which may depend on the small parameters $\epA$ and $\epM$,
are uniformly bounded in $H^n$ and satisfy \eqref{eq:divb0} and \eqref{eq:wellprep}.
  \begin{description}
  \item[Uniform Existence.] Under the above conditions
     there
  exist fixed positive  $T$ and $K$ such that for $(\epA,\epM)$ satisfying \eqref{eq:epAepMgen}
  the solution to \eqref{MHDsy}   having
  the initial data $\vV^0$ exists for $0\le t\le T$ and satisfies
    \begin{equation}\label{eq:uniform}
    \sup_{0\le t\le T}  \left[ \|\vV\|_{n}+\| \vV_t\|_{0}+\sum_{j=1}^n \epM^{j-1}\left(\min(\tfrac{\epA}{\epM},1)\right)^{n-1}
    \|\pt^j  \vV\|_{{n-j}}  \right]\le K.
  \end{equation}

\item[Convergence and Limit.]
  Assume in addition that the normalization \eqref{eq:pp1} holds and
  that as $(\epA,\epM)$ satisfying \eqref{eq:epAepMgen}  tend to zero, their ratio $\frac{\epA}{\epM}$ converges to some value
$\mu_{\lim}$ and the initial data $\vV^0$ converges in $H^n$ to $\vVlim^0$.
Then the solution $\vV=(r,\vu=(\uh,u_3),\vb=(\bh,b_3))$ of the MHD system \eqref{MHDsy} with initial data $\vV^0$ converges in $C^0([0,T];H^{n-\alpha})$ for every
$\alpha>0$. Its limit is independent of $z$, and is  the unique
solution~$\vVlim=(\rlim,(\uhlim,\uvlim), (\bhlim,\bvlim))$ of the limit system
\begin{subequations}\label{limit:CJS0}
\begin{align}
\label{eq:limit:r}
(1+\mu_{\lim}^2)&\left[ \partial_t \rlim+(\uhlim\cnh)\rlim\right]+\mu_{\lim}(\bhlim\cnh) \uvlim=0,
  \\
\label{eq:limit:uh}
\cPdh \big(&\partial_t\uhlim+(\uhlim\cnh)\uhlim-(\bhlim\cnh)\bhlim\big)=0,&  \nch\uhlim=0,\\
  \label{eq:limit:u3}
 &\partial_t\uvlim+(\uhlim\cnh)\uvlim+\mu_{\lim}(\bhlim\cnh)\rlim=0,  
  \\
\label{eq:limit:bh}
&\partial_t\bhlim+(\uhlim\cnh)\bhlim-(\bhlim\cnh) \uhlim=0,&  \nch\bhlim=0,\\
  \label{eq:limit:b3}
&\bvlim=\avv\bvlim^0-\mu_{\lim}(\rlim-\avv\rlim^0)
\end{align}
\end{subequations}
having initial data $\vVlim^0$.

\item[Rate of Convergence.]  In addition to the original assumptions and the additional assumptions of the convergence part,
  assume that the initial data for the MHD system have the more
specific form \eqref{eq:initfis}, \eqref{eq:divb00}, \eqref{eq:initfixed} and that \eqref{eq:epAltepM} holds.
Then there is a constant $c$ independent of $\epA$ and $\epM$ such that for all $t\in[0,T]$,
    \begin{subequations}
      \label{eq:errest}
      \begin{align}
   \label{eq:like:fastest2}
        &\begin{aligned}\|(&1-\cPdh\az)\uh\big\|_{j}+\|(1-\az)\bh\|_{j}
          +\|b_3-\avv \bvlim^0+\mu (\az r-\avv \rlim^0)\|_{j}\le c\, \epM^{1-(j-1)\nu}, \quad j=0,\ldots,n-1,
        \end{aligned} 
        \\
  \label{eq:like:slowest1}
  &\|\cPdh\az \uh-\uhlim\|_{n-2}+\|\az\bh-\bhlim\|_{n-2}\le c\, \epM,
        \\
\label{eq:likeintest}
    &\|(1-\az)r\|_{j}+\|(1-\az)u_3\|_{j}\le c\, \epM^{1-(j-1)\nu}, \quad j=1,\ldots,n-1, 
 \\ \label{eq:like:slowru3est2}
  &\|\az r-\rlim\|_{n-2}+\|\az u_3-\uvlim\|_{n-2}\le c \left[\epM^{1-\max(n-5,0)\nu}+|\mu-\mu_{\lim}|\right].
 \\\intertext{Moreover, there exist $O(1)$ correctors $(\rcor,\uvcor)$ defined in \eqref{eq:rcor}--\eqref{eq:uvcor} such that}
        \label{eq:like:slowestcor2}
        &\begin{aligned}
          \|\az r-(\rlim+\tfrac{\mu-\mu_{\lim}}{1+\mu^2}\rcor)\|_{n-2}&+\|\az u_3-(\uvlim+(\mu-\mu_{\lim})\uvcor)\|_{n-2}
          \le c\,  \epM^{1-\max(n-5,0)\nu}.
        \end{aligned}
      \end{align}
    \end{subequations}
\end{description}
\end{thm}
The uniform existence part of Theorem~\ref{thm:theom1} is a special case of the
corresponding general result for systems \eqref{eq:general} stated in Lemma~\ref{lem:bound} and proven in \S\ref{sec:bound}, which is an improvement of
\cite[Theorem 3.6]{CJS17}. The convergence part of Theorem~\ref{thm:theom1} will be proven in \S\ref{sec:conv}. The convergence-rate estimates in~\eqref{eq:errest} 
are direct consequences of the estimates \eqref{eq:fastest}, \eqref{eq:intestdyn}, \eqref{eq:slowest1}, \eqref{eq:slowru3est2}, and
\eqref{eq:slowest2b}  proven in \S\ref{sec:rate}.

Under the scaling \eqref{eq:epAepMgen}
all powers
of $\epM$ appearing in \eqref{eq:errest} are positive, so a nontrivial rate of convergence is obtained over the full range of allowed
values of $\nu$, in all the norms listed in the theorem. 
The corrector estimate \eqref{eq:like:slowestcor2} has been included because
\eqref{eq:epAepMgen}\textendash\eqref{eq:nudef} imply that
estimate \eqref{eq:like:slowru3est2} is much weaker than the other estimates
in \eqref{eq:errest} in the main case of interest in which $\mu_{\lim}=0$.
The improved estimate \eqref{eq:like:slowestcor2} involving the
corrector shows that \eqref{eq:like:slowru3est2} is in fact sharp in this case, and gives a formula for the principal error term.

For well-prepared initial data like that considered here, the convergence rate
for solutions of two-scale systems is typically first order in the single
small parameter \cite{Sch:rate,Sch:limits}. That result depends crucially on
the uniform boundedness of the first time derivative of solutions being
propagated for positive time, which does \emph{not} generally hold for
three-scale systems~\eqref{eq:general} \cite[\S2]{CJS17}. We use the
time-integration method developed in \cite{Cheng:SIMA:2012,Cheng:Mach} to
mitigate the effect of the lack of uniform boundedness of first time
derivatives. Moreover, the estimates in \eqref{eq:errest} \emph{are} $O(\epM)$
for those components and norms for which uniform boundedness of the first time
derivative holds, and tend, except for \eqref{eq:like:slowestcor2}, to the
two-scale $O(\epM)$ convergence rate as $\nu\to0$, which makes
\eqref{eq:general} tend to a two-scale system. Obtaining that asymptotic
consistency is only possible on account of the improvement \eqref{eq:uniform}
over the estimate $\epM\|\vV_t\|_{{n-1}}\le c$ in \cite{CJS17}.

The $H^j$ estimates \eqref{eq:like:fastest2}
and \eqref{eq:likeintest} for intermediate values of $j$ are obtained using interpolation. Those estimates
are the starting point for an improved estimate for $z$-averages of products
derived in \S\ref{sec:azprod}, which is used in the proof of \eqref{eq:like:slowest1},
and should be useful more generally.

One of the main techniques used in the proof of \eqref{eq:errest} is partitioning the solution into fast, intermediate, and slow modes,
and analyzing each mode separately. In contrast to the two-scale case (e.g. \cite[\S2]{Cheng:Mach}), 
the exact eigenspaces of the large operator $\fr\epA \cLA+\fr\epM \cLM$ having eigenvalues of sizes strictly $O(\fr\epA)$,
$O(\fr\epM)$, and $o(\fr\epM)$ depend on the parameter~$\mu$. For simplicity,
in \S\ref{sec:fis} we define the fast, intermediate, and slow
modes to be the fixed projections onto the limits as $\mu\to0$ of the exact
eigenspaces. These projections are also used to
determine the appropriate form of the initial data used in the convergence rate part of Theorem~\ref{thm:theom1}.
However, certain estimates in
\S\ref{sec:rate} for the intermediate modes require the use of the exact
$\mu$-dependent modes.

The only previous convergence rate result we know for evolutionary PDEs with
two parameters tending to zero independently appears in~\cite{MR3723323},
which considered the low Rossby and magnetic Reynolds number limit of the
stochastically-forced viscous incompressible rotating MHD system.  Compared to
the hyperbolic system considered here the deterministic case without forcing
\cite[p. 4444]{MR3723323} of their system has many simplifying features,
including the absence of a matrix multiplying the time derivatives,
which eliminates the need for restriction~\eqref{eq:epAepMgen}, the
presence of a closed $L^2$ energy estimate, and the presence of regularizing
viscous terms whose diffusivity rates tend to infinity as the small parameters
tend to zero, which yields a highly parabolic system that induces smoothing
when the large operator is used to determine the fast components in terms of
the slow component.

\section{Analysis of the Large Operator}\label{sec:fis}

Following \cite[\S4]{CJS17} but without treating each Fourier mode separately,
let $\mathbb{P}^0$ denote the $L^2$-orthogonal projection operator onto the
nullspace of $\cLA$, and let $\mathbb{P}^1$ denote the $L^2$-orthogonal
projection operator onto the nullspace of $\mathbb{P}^0\cLM\mathbb{P}^0$.
Then $R(I-\mathbb{P}^0)\subseteq
N(\mathbb{P}^0\cLM\mathbb{P}^0)=R(\mathbb{P}^1)=N(I-\mathbb{P}^1)$, so
$R(I-\mathbb{P}^0)\perp R(I-\mathbb{P}^1)$, and hence
\begin{equation}\label{eq:orthproj}
  (I-\mathbb{P}^0)(I-\mathbb{P}^1)=0=(I-\mathbb{P}^1)(I-\mathbb{P}^0).
\end{equation}
Expanding the factors in \eqref{eq:orthproj} shows that 
$\mathbb{P}^0\mathbb{P}^1=\mathbb{P}^1\mathbb{P}^0$,
which implies that $\mathbb{P}\eqdef\mathbb{P}^0\mathbb{P}^1$ is an orthogonal projection operator satisfying
$\mathbb{P}(I-\mathbb{P}^j)=0=(I-\mathbb{P}^j)\mathbb{P}$ for ${j\in\{0,1\}}$.
Moreover, \eqref{eq:orthproj} and the definition of $\mathbb{P}$ yield
$(I-\mathbb{P}^0)+(I-\mathbb{P}^1)+\mathbb{P}=I+(I-\mathbb{P}^0)(I-\mathbb{P}^1)=I$,
which shows that the sum of the fast, intermediate, and slow modes defined by
\begin{equation}\label{eq:modef}
  \vV^F\eqdef (I-\mathbb{P}^0)\vV, \quad \vV^I\eqdef (I-\mathbb{P}^1)\vV, \quad \vV^S\eqdef \mathbb{P}\vV
\end{equation}  
satisfies $\vV^F+\vV^I+\vV^S=\vV$.
These modes are the limits as $\mu\to0$ of the direct sums of the eigenspaces
of $\cLA+\mu\cLM$ whose eigenvalues are strictly $O(1)$, $O(\mu)$, and
$o(\mu)$, respectively \cite{MR678094}, \cite[\S4]{CJS17}.
Moreover, since the $\mathbb{P}^j$ are orthogonal projections onto the null spaces of constant-coefficient differential
operators they commute with derivatives, and hence are also orthogonal in any $H^k$. 
Therefore estimates for the full solution obtained by combining estimates for
each mode are as sharp as the component estimates, modulo constant factors.

The above results do not depend on the particular form of the operators~$\cLA$
and~$\cLM$. We now calculate the projections and modes for the MHD
system \eqref{MHDsy}.  For brevity, we restrict consideration to
$\vV=(r,\vu,\vb)$ satisfying $\nc\vb=0$, which causes no difficulties since we only
consider initial data and solutions satisfying that constraint.
Recall that $\cPdh$, $\az$, and $\avv$
were defined in \eqref{def:cPdh0} and~\eqref{eq:avvdef}.
\begin{lem}\label{lem:bdiv0}
  Assume that the spatial domain is $\mT^3$ and that $\vV=(r,\vu,\vb)$ where $\vb$ satisfies $\nc\vb=0$. Then
  \begin{enumerate}
  \item   $\cLA \vV=\boldzero$ iff   $\pz \uh=\boldzero=\pz \bh$, $\nch \uh=0$, and $b_3=\avv b_3$.

\item  $\left(\fr{\epA}\cLA +\fr{\epM}\cLM \right)\vV=\boldzero$ iff 
  \begin{equation}
    \label{eq:lalmz}
 \pz \vV=\boldzero, \quad \nch\uh=0,\quad b_3=\avv b_3
 -\tfrac{\epA}{\epM}\left( r-\avv r\right).
  \end{equation}

\item The formulas for the projections are
  \begin{equation}\label{def:P0}
  \begin{split}
\mathbb P^0=\begin{pmatrix}
         1& & \\&\left(\begin{smallmatrix}
         \cPdh \az I_{2\times2}
        & \\
         & 1
         \end{smallmatrix}\right)
         & \\& & \left(\begin{smallmatrix} \az I_{2\times2} 
         \\ &\avv
         \end{smallmatrix}\right)
       \end{pmatrix},
       \quad
       \mathbb P^1=\begin{pmatrix}
         \az& & \\&\left(\begin{smallmatrix}
         I_{2\times 2}& \\
         & \az
         \end{smallmatrix}\right)
         & \\& & I_{3\times 3}
       \end{pmatrix},
\\
\mathbb{P}=\begin{pmatrix}
         \az& & \\&\left(\begin{smallmatrix}
         \cPdh \az I_{2\times2}
        \\
         & \az
         \end{smallmatrix}\right)
         & \\& & \left(\begin{smallmatrix}
           \az I_{2\times2} 
        \\
         &\avv
         \end{smallmatrix}\right)
       \end{pmatrix},
\end{split}
\end{equation}
where all missing entries vanish.

\item All eigenvalues of $\cLA+\mu\cLM$ that are $o(\mu)$ are identically zero.

\item Using the notations $\vV^\ell=(r^\ell,\vu^\ell,   \vb^\ell)$ for $\ell\in\{F,I,S\}$ and
  $\vw^\ell=(\wh^\ell,w_3^\ell)$ for $w\in\{\vu,\vb\}$, the formulas for the fast, intermediate and slow modes are
\begin{equation}
  \label{eq:modeforms}
  \begin{aligned}
    &\begin{pmatrix}
      r^F&\vu^F & \vb^F
        \end{pmatrix}=
  \begin{pmatrix}
    0&
    \begin{pmatrix}
      (1-\az\cPdh)\uh\\0
    \end{pmatrix}&
    \begin{pmatrix}
      (1-\az)\bh\\(1-\avv)b_3
    \end{pmatrix}
  \end{pmatrix},
  \\&
  \begin{pmatrix}
    r^I&\vu^I&\vb^I
  \end{pmatrix}=
  \begin{pmatrix}
    (1-\az)r&
    \begin{pmatrix}
      0_2\\(1-\az)u_3
    \end{pmatrix}&
\boldzero_3
  \end{pmatrix},
  \\&\begin{pmatrix}
    r^S&\vu^S&\vb^S
  \end{pmatrix}=
  \begin{pmatrix}
    \az r&
    \begin{pmatrix}
      \cPdh\az \uh\\\az u_3
    \end{pmatrix}&
    \begin{pmatrix}
      \az \bh\\\avv b_3
    \end{pmatrix}
  \end{pmatrix}.
\end{aligned}
\end{equation}
In particular,
\begin{equation}
  \label{eq:balpdivfree}
  \nc\vb^F=0,\quad \nc\vb^I=0,\quad \nc\vb^S=0,
\end{equation}
\begin{equation}
  \label{eq:ind:z}
  \text{$\vV^S$ is independent of $z$,}
\end{equation}
and 
\begin{equation}
  \label{eq:h:div:0}
  \nch\uh^S=0=\nch\bh^S.
\end{equation}

\item For any nonnegative integer $j$, there are constants $c_1$ and $c_2$ such that
  \begin{equation}
    \label{eq:fifromlarge}
    \begin{aligned}
      c_1\|(\cLA+\mu\cLM)\vV\|_j
      &\le \|\pz \uh^F\|_j+\|\nch \uh^F\|_{j}+\|\pz(\bh^F-\mu\Deltainv\nabh\pz r^I)\|_j
      \\&\qquad+\|b_3^F+\mu\Deltainv\Delta_h r\|_{j+1}
            +  \mu[\|\pz r^I\|_{j}+\|\pz u_3^I\|_j] 
      \\&\le c_2\|(\cLA+\mu\cLM)\vV\|_j.
    \end{aligned}
  \end{equation}
 \end{enumerate}
\end{lem}

\begin{proof}
  Applying $\cLA$ to $\vV$ yields $\left( 0, \left(\begin{smallmatrix}  \pz \bh-\nabh b_3\\0  \end{smallmatrix}\right),
  \left(\begin{smallmatrix}  \pz \uh\\-\nch \uh   \end{smallmatrix}\right)\right)$. Hence the last part of $\cLA\vV$ vanishes iff
  $\pz \uh=\boldzero$ and $\nch \uh=0$. Taking the horizontal divergence of
  the second component of $\cLA$ and using the fact that $\nc\vb=0$ yields
  \begin{equation}\label{eq:getlapb3}
    \nch (\pz \bh-\nabh b_3)= -\Delta b_3.
  \end{equation}
Hence if $\cLA\vV=0$ then $b_3$ is a constant, i.e., $b_3=\avv b_3$, which implies further that $\pz \bh=\nabh b_3=\boldzero$.
 On the other hand, when those conditions hold then each term in the second part of $\cLA\vV$ vanishes.

Similarly, 
\begin{equation}\label{eq:lalmv}
\left(\fr{\epA}\cLA +\fr{\epM}\cLM \right)\vV=
\left(
\begin{smallmatrix}
  -\fr{\epM}\nc\vu\\\left(
  \begin{smallmatrix}
   \fr{\epA}\left( -\nabh b_3+\pz \bh-\frac{\epA}{\epM}\nabh r\right) \\ -\fr{\epM}\pz r 
  \end{smallmatrix}
\right)
\\
\fr{\epA}\left(
\begin{smallmatrix}
  \pz \uh \\ -\nch \uh
\end{smallmatrix}
\right)
\end{smallmatrix}
\right).
\end{equation}
The last component of the second part vanishes iff
\begin{equation}\label{eq:rz0} \pz r=0,
\end{equation}
and since $\pz u_3=\nc\vu-\nch\uh$, the first and last parts vanish iff $\nch\uh=0$ and
$\pz\vu=\boldzero$. Next, use \eqref{eq:getlapb3} and \eqref{eq:rz0}
to write the horizontal divergence of the second part  as
  \begin{equation}\label{eq:delb3mudelhr}
    \nch \left( -\nabh b_3+\pz \bh-\frac{\epA}{\epM}\nabh r\right)=-(\Delta
    b_3+\tfrac{\epA}{\epM}\Delta_{\mathsf h} r)=
    \Delta( b_3+ \tfrac{\epA}{\epM} r),
  \end{equation}
  which  vanishes iff  the last equation of \eqref{eq:lalmz} holds.
  That equation together with \eqref{eq:rz0}  implies that $\pz b_3=0$. 
  Finally, the last equation of \eqref{eq:lalmz}  also shows that the horizontal components of the
  second part of \eqref{eq:lalmv} vanish iff  $\pz\bh=\boldzero$.

  The formula for $\mathbb{P}^0$ follows from the conditions for
  $\cLA\vV$ to vanish in the first part of the lemma, and that formula
  together with the formula for $\cLM$ in \eqref{eq:largeops} 
    implies that
  \begin{equation}\label{eq:p0mp0}
    \mathbb{P}^0\cLM\mathbb{P}^0\vV=-\left(\pz u_3, \left( \begin{smallmatrix}0\\\pz r\end{smallmatrix}\right),0\right).
  \end{equation}
  Formula \eqref{eq:p0mp0}
  implies the formula for $\mathbb{P}^1$, and the formulas for the
  $\mathbb{P}^j$ yield $\mathbb{P}$.

The conditions \eqref{eq:lalmz} for $(\cLA+\mu\cLM)\vV$ to vanish differ from
the conditions to belong to the null space of $\mathbb{P}$
only by adding an $O(\mu)$ term to the formula for $b_3$. Hence the rank of the restriction to any Fourier mode  of $N(\cLA+\mu\cLM)$
equals the rank of the restriction to that mode of $\mathbb{P}$, which in turn equals the dimension of 
the direct sum of all eigenspaces of $\cLA+\mu\cLM$ in that Fourier mode having eigenvalues of size~$o(\mu)$ \cite{MR678094}, \cite[\S4]{CJS17}.
Hence all such eigenvalues vanish identically.

The formulas \eqref{eq:modeforms} for the modes follow from their definition
\eqref{eq:modef} and formula~\eqref{def:P0} for the projections.  Formula
\eqref{eq:modeforms} for $\vb^S$ plus the fact that
$\az\vb$ is independent of $z$  imply that $\nc\vb^S=0$, and trivially
$\nc\vb^I=0$, which implies the rest of \eqref{eq:balpdivfree}. Each component
of $\vV^S$ contains $\az$ or $\avv$, so \eqref{eq:ind:z} holds. The left
equation in \eqref{eq:h:div:0} follows from the presence of the operator
$\cPdh$ in the formula for $\uh^S$, while the right equation there follows
from \eqref{eq:balpdivfree}--\eqref{eq:ind:z}.

The formula for $\uh^F$ in \eqref{eq:modeforms}  and the formula \eqref{eq:lalmv} for $(\cLA+\mu\cLM)\vV$ yield
\begin{align}\label{eq:1mazuhf}
    \|\pz\uh^F\|_j=\|\pz \uh\|_j\le c\|(\cLA+\mu\cLM)\vV\|_j,
  \\
  \label{eq:divhuhf}
  \|\nch\uh^F\|_j=\|\nch\uh\|_j\le c\|(\cLA+\mu\cLM)\vV\|_j.
\end{align}
By the ellipticity of $\Delta$, formula \eqref{eq:modeforms} for $b_3^F$, the first identity in \eqref{eq:delb3mudelhr} and
\eqref{eq:lalmv},
\begin{equation}\label{eq:b3fest}
  \begin{aligned}
    \|b_3^F+\mu\Deltainv\Delta_h r\|_{j+1}&=\|\Deltainv[\Delta (1-\avv)b_3+\mu\Delta_h r\|_{j+1}\le c\|\Delta b_3+\mu\Delta_h r\|_{j-1}
  \\&=c \|\nch (\nabh b_3+\mu\nabh r-\pz \bh)\|_{j-1}\le c \|(\cLA+\mu\cLM)\vV\|_j
\end{aligned}
\end{equation}
Also, combining  
formula \eqref{eq:modeforms} for $\bh^F$, formula
\eqref{eq:lalmv} for $\cLA+\mu\cLM$, and \eqref{eq:b3fest} yields
\begin{equation}\label{eq:bhf}
  \begin{aligned}
    \|\pz(\bh^F-\mu\Deltainv\nabh\pz r^I)\|_{j}
    &= \|\pz(\bh-\mu\Deltainv\nabh\pz r)\|_{j}
    \\&=\|\pz \bh-\mu\nabh\Deltainv(\Delta r-\Delta_h r)\|_{j}
  \\&=\|(\pz \bh-\mu\nabh r-\nabh b_3)+\nabh(b_3+\mu\Deltainv\Delta_h r)\|_{j}
  \\&\le \|(\cLA+\mu\cLM)\vV\|_{j}+\|b_3^F+\mu\Deltainv\Delta_h r\|_{j+1}\le c\|(\cLA+\mu\cLM)\vV\|_{j}.
\end{aligned}
\end{equation}
 By the formula
\eqref{eq:modeforms} defining the modes and formula \eqref{eq:lalmv} for $\cLA+\mu\cLM$,
\begin{align}
  \label{eq:staticI1}
  \mu\|\pz r^I\|_j&=\mu\|\pz r\|_j\le c\|(\cLA+\mu\cLM)\vV\|_{j},
\\ \label{eq:staticI2}                    
  \mu\|\pz u_3^I\|_j&=\mu\|\pz u_3\|_j=\mu\|\nc\vu-\nch\uh\|_j\le c\|(\cLA+\mu\cLM)\vV\|_{j}
\end{align}
Combining \eqref{eq:1mazuhf}\textendash\eqref{eq:staticI2} yields the right inequality
in \eqref{eq:fifromlarge}, and the expressions estimated there yield all the terms in $(\cLA+\mu\cLM)\vV$ so the left inequality
also holds.
\end{proof}

Solving the PDE for $(\cLA+\mu\cLM)\vV$ and using the bounds \eqref{eq:uniform} to estimate the result yields a bound for the $H^{n-1}$ norm 
of that expression, which by \eqref{eq:fifromlarge} implies ``static'' estimates for the fast and intermediate modes, which will be written explicitly
in \S\ref{sec:rate}.
  However, we also need to obtain ``dynamic'' estimates for the intermediate and slow
  modes via differential inequalities.  To do so we cannot use the limit modes defined above.
  The exact slow eigenspace of the zero eigenvalue of $\cLA+\mu\cLM$ was determined in
  Lemma~\ref{lem:bdiv0}.  In the next lemma we obtain formulas for $7$-vectors of Fourier multiplier operators
  $\vV_{\alpha}$ and $\vV_{\beta}$  and the self-adjoint Fourier multiplier operator $\mathcal{Q}$, such that
  \begin{equation}\label{eq:vabq}
    (\cLA+\mu\cLM)\vV_{\alpha}\cdot \widetilde{\vV}+\mu \mathcal{Q}\pz\vV_{\beta}\cdot\widetilde{\vV}=0=
     (\cLA+\mu\cLM)\vV_{\beta}\cdot\widetilde{\vV}+\mu \mathcal{Q}\pz\vV_{\alpha}\cdot\widetilde{\vV}   
  \end{equation}
  for every vector-valued function~$\widetilde{\vV}$.
  Equations \eqref{eq:vabq} imply that for every Fourier mode $(k,l,m)$ the linear combinations $(\vV_\alpha\pm \vV_{\beta})e^{i(kx+ly+mz)}$ are eigenfunctions
  of the operator $\cLA+\mu\cLM$ with purely imaginary or zero eigenvalues  $\mp im \mu \widehat{\mathcal{Q}}$ where
  \begin{equation}\label{eq:qhat} \widehat{\mathcal{Q}}\eqdef e^{-i(kx+ly+mz)}\left[ \mathcal{Q} e^{i(kx+ly+mz)}\right],
  \end{equation}
  i.e.,
  they yield the exact  $\mu$-dependent intermediate eigenspaces.
However, in \S\ref{sec:rate} we will obtain dynamic estimates for  $(1-\az)\vV_\alpha\cdot\vV$ and $(1-\az)\vV_{\beta}\cdot\vV$
 rather than $(\vV_\alpha\pm \vV_{\beta})\cdot\vV$, to reduce disruption to the structure of the rest of the PDE. The operator $1-\az$ is
 applied because the eigenvalues $\pm im \mu \mathcal{Q}$ are only of size $\mu$ when $m\ne0$. Moreover, since $\nc\vb=0$  we will replace $\vV_\alpha$ by
the variant $\vValpnodiv$ defined in \eqref{eq:vabdef} that omits the gradient term in the magnetic field component. 

  Since the limits as $\mu\to0$ of $\vV_{\alpha}$ and $\vV_{\beta}$ should belong to the intermediate mode
  defined in \eqref{eq:modeforms}, trying various perturbations leads to the ansatz
    \begin{equation}
      \label{eq:vabdef}
      \begin{aligned}
        \vV_{\alpha}&\eqdef \vV_\alpha^{\textnormal{\tiny div}}+\vValpnodiv=
        \begin{pmatrix}
      0\\0_3\\\mu \mathcal{A}\Deltainv \pz\nabla
      \vphantom{\left(\begin{smallmatrix}  0\\0\\1 \end{smallmatrix}\right)}
    \end{pmatrix}
    +
    \begin{pmatrix}
      1\\0_3\\-\mu\left(
      \begin{smallmatrix}
        0\\0\\1
      \end{smallmatrix}\right)+\mu^3 \left(
      \begin{smallmatrix}
        0\\0\\\mathcal{B}
      \end{smallmatrix}\right)
    \end{pmatrix},
\\
    \vV_{\beta}&\eqdef
    \begin{pmatrix}
      0\\\left(
      \begin{smallmatrix}
        \mu^2 \mathcal{C} \Deltainv \pz\nabla_h\\\mathcal{D}
      \end{smallmatrix}
\right)\\0_3
    \end{pmatrix},
  \end{aligned}
\end{equation}
where the vectors have been normalized by setting the first component of $\vV_\alpha$ to the identity operator~$1$, and factors of $\Deltainv$ have
been included so that if $(\mathcal{A},\mathcal{B},\mathcal{C},\mathcal{D})$ are all homogeneous order zero Fourier multipliers then all components of
$\vV_{\alpha}$ and $\vV_{\beta}$ will also be Fourier multipliers of order zero.  Furthermore, in the limit as $\mu\to0$ the operator
$-\pz \mathcal{Q}$ should tend to the operator $-\pz$ appearing in $\mathbb{P}^0\cLM\mathbb{P}^0$ in \eqref{eq:p0mp0}, i.e., $\mathcal{Q}$ should
tend to one.

\begin{lem}\label{lem:exacti}
  The vectors $\vV_\alpha$ and $\vV_{\beta}$ defined in \eqref{eq:vabdef} satisfy \eqref{eq:vabq} with
  \begin{equation}
    \label{eq:qdef}
    \begin{aligned}
      &\mathcal{Q}=\sqrt{2\left( 1+\mu^2+\sqrt{( 1+\mu^2)^2-4\mu^2\pz^2\Deltainv}\right)^{-1}},
      \quad\text{i.e.,}\quad
    \widehat{\mathcal{Q}}=\tfrac{\sqrt2}{\sqrt{1+\mu^2+\sqrt{(1+\mu^2)^2-4\mu^2\frac{m^2}{k^2+l^2+m^2}}}},
  \end{aligned}
\end{equation}
    provided that 
  \begin{equation}
    \label{eq:abphipsi}
    \mathcal{C}=-\mathcal{Q} \left( 1-\mu^2\pz^2\Deltainv \mathcal{Q}^2\right)^{-1}, \quad \mathcal{A}=-\mathcal{C}\mathcal{Q}^{-1},
    \quad \mathcal{B}=\pz^2\Deltainv \mathcal{C}\mathcal{Q},  \quad \mathcal{D}=\mathcal{Q}^{-1}.
  \end{equation}
\end{lem}

\begin{proof}
  Substituting \eqref{eq:vabdef} into \eqref{eq:vabq} yields 
  \begin{equation}
    \label{eq:needz}
    \begin{aligned}
    &\begin{pmatrix}
      0\\\left(\begin{smallmatrix}-\mu^3 (\mathcal{B}-\mathcal{C} \mathcal{Q}\pz^2\Deltainv)\nabh\\-\mu\pz(1-\mathcal{D}\mathcal{Q})\end{smallmatrix}\right)\\0_3
    \end{pmatrix}=\boldzero=
    \begin{pmatrix}
      -\mu\pz(\mathcal{D}-\mathcal{Q}+\mu^2\mathcal{C}\Delta_h\Deltainv)\\0_3\\\mu^2\pz^2\Deltainv(\mathcal{C}+\mathcal{A}\mathcal{Q})\nabla
      -\left(\begin{smallmatrix}0\\0\\\mu^2\pz(\mathcal{C}+\mathcal{Q}-\mu^2\mathcal{B}\mathcal{Q})\end{smallmatrix}\right)
    \end{pmatrix},
  \end{aligned}
\end{equation}
  which will hold provided that
  \begin{equation}
    \label{eq:needz2}
    \begin{aligned}
    &\mathcal{C}+\mathcal{A}\mathcal{Q}=0,\quad    \mathcal{D}\mathcal{Q}=1,\quad   \mathcal{B}=\mathcal{C}\mathcal{Q}\pz^2\Deltainv,
    \quad  \mathcal{C}+\mathcal{Q}=\mu^2\mathcal{B}\mathcal{Q},
    \quad
    \mathcal{D}-\mathcal{Q}+\mu^2\Delta_h\Deltainv \mathcal{C}=0.
  \end{aligned}
\end{equation}
Solving the first three equations  in \eqref{eq:needz2} for $\mathcal{A}$, $\mathcal{D}$, and $\mathcal{B}$ yields the formulas for those operators
claimed in \eqref{eq:abphipsi}. Substituting those  formulas  into the fourth equation in \eqref{eq:needz2} and solving the result for $\mathcal{C}$
yields the formula for that operator in \eqref{eq:abphipsi}. Substituting the formulas obtained so far 
into the last equation in \eqref{eq:needz2} yields
 $\mathcal{Q}^{-1}-\mathcal{Q} -\mu^2\Delta_h\Deltainv  \mathcal{Q}(1-\mu^2\pz^2\Deltainv \mathcal{Q}^2 )^{-1}=0$,
 whose only solution tending to one as $\mu\to0$ is~\eqref{eq:qdef}.
\end{proof}

We now turn to explicating the form that initial data satisfying \eqref{eq:wellprep} takes.
\begin{lem}\label{lem:initform}
  Initial data $\vV^0=(r^0,\vu^0,\vb^0)$
will be uniformly bounded in $H^n$ and satisfy the constraint $\nc\vb^0=0$ and the condition \eqref{eq:wellprep} iff it has the form
  \begin{equation}
    \label{eq:initfis}
    \begin{aligned}
      r^0&=r^{0,S}+\epM  r^{0,I}, \quad \uh^0=\uh^{0,S}+\epA \uh^{0,F}, \; u_3^0=u_3^{0,S}+\epM u_3^{0,I},
      \\ \bh^0&=\bh^{0,S}+\epA \bh^{0,F}, \;
      b_3^0=\avv b_3^0 -\mu (1-\avv) r^{0,S}+\epA b_3^{0,F}, 
    \end{aligned}
  \end{equation}
  where every term $w^{0,\ell}$  with $w\in \{r,\uh,u_3,\bh,b_3\}$ and $\ell\in\{F,I,S\}$ has the form specified for the $w$ component
  of the $\ell$ mode in \eqref{eq:modeforms},
  and may depend on $(\epA,\epM)$ but satisfies
  \begin{equation}
    \label{eq:initfisbnd}
    \begin{aligned}
    \|r^{0,S}\|_n&+\left[ \|\pz r^{0,I}\|_{n-1}+\epM\|r^{0,I}\|_n\right]
    \\&\;
    +\|\uh^{0,S}\|_n+\left[\|\pz \uh^{0,F}\|_{n-1}+\|\nch\uh^{0,F}\|_n+\epA\|\uh^{0,F}\|_n\right]
    +\|u_3^{0,S}\|_n+\left[\|\pz u_3^{0,I}\|_{n-1}+\epM\| u_3^{0,I}\|_n\right]
    \\&\;
    +\|\bh^{0,S}\|_n+\left[\|\pz \bh^{0,F}\|_{n-1}+\epA\|\bh^{0,F}\|_n\right]
    +|\avv b_3^0|+\|b_3^{0,F}+\Deltainv\Delta_h r^{0,I}\|_n\le c
  \end{aligned}
\end{equation}
uniformly in those parameters, and
\begin{equation}\label{eq:divb00}
  \nch\bh^{0,S}=0=\nch \bh^{0,F}+\pz b_3^{0,F}.
\end{equation}
\end{lem}

\begin{proof}
  Since the terms in \eqref{eq:initfis} are allowed to depend on $\epA$ and $\epM$, that formula simply expresses the separation of the
  initial data into fast, intermediate, and slow modes, with the factors of $\epA$ and $\epM$  and the inclusion of the
  specific term $-\mu(1-\avv)r^{0,S}$   being purely for later convenience.
  By Lemma~\ref{lem:bdiv0}, the condition $\nc\vb=0$ implies that each mode is divergence-free, and the conditions $\nc\vb^\ell=0$ for $\ell\in\{F,I,S\}$
  clearly imply that $\nc\vb=0$, so for initial data of the form~\eqref{eq:initfis} the conditions \eqref{eq:divb00}
  are equivalent to the assumed condition that $\nc\vb^0=0$.  
  
  Since, as shown above, the square of the $H^n$ norm of $\vV^0$ equals the sum of the
  squares of the $H^n$ norms of its modes, the assumed uniform boundedness of $\|\vV^0\|_n$ is equivalent to
  \begin{equation}\label{eq:hninit}
    \begin{aligned}
      \|r^{0,S}\|_n&+\epM\|r^{0,I}\|_n+\|\uh^{0,S}\|_n+\epA\|\uh^{0,F}\|_n+\|u_3^{0,S}\|_n+\epM\| u_3^{0,I}\|_n
      \\&+\|\bh^{0,S}\|_n+\epA\|\bh^{0,F}\|_n + |\avv b_3^0|+\|-\mu(1-\avv)r^{0,S}+\epA b_3^{0,F}\|_n\le c.
    \end{aligned}
  \end{equation}
  For $k=n-1$ the sum of terms estimated in \eqref{eq:fifromlarge}  is equivalent to the $H^{n-1}$ norm of $(\cLA+\mu\cLM)\vV$.
Hence the assumed uniform boundedness of $\fr{\epA}\|(\cLA+\mu\cLM)\vV^0\|_{n-1}$ becomes, for initial data $\vV^0$ of the form \eqref{eq:initfis}
satisfying \eqref{eq:divb00},
\begin{equation}
  \label{eq:epainvinit}
  \begin{aligned}
    \|\pz \uh^{0,F}\|_{n-1}&+\|\nch \uh^{0,F}\|_{n}+\|\pz(\bh^{0,F}-\Deltainv\nabh\pz r^{0,I})\|_{n-1}
    +  \|\pz r^{0,I}\|_{n-1}
    \\&+\|\pz u_3^{0,I}\|_{n-1} 
  +\|-\epM^{-1}(1-\avv)r^{0,S}+b_3^{0,F}
  +\epM^{-1}\Deltainv\Delta_h r^{0,S}
  +\Deltainv\Delta_h r^{0,I}\|_{n}
    \le c.
  \end{aligned}
\end{equation}
Since the expression $-\mu(1-\avv)r^{0,S}$ appearing in the last term in \eqref{eq:hninit} can be estimated by the first term there it can be
omitted from that last term, leaving $\epA\|b_3^{0,F}\|_n$.  Similarly, the expression $-\pz(\Deltainv\Delta_h)\pz r^{0,I}$ in the third term of
\eqref{eq:epainvinit} can be omitted.  Also the expression
$-\epM^{-1}(1-\avv)r^{0,S}+\epM^{-1}\Deltainv\Delta_h r^{0,S}$ in the last term of \eqref{eq:epainvinit} vanishes identically because $r^{0,S}$
is independent of $z$, leaving $\|b_3^{0,F}+\Deltainv\Delta_h r^{0,I}\|_{n}$.  The uniform bounds for that term and for $\epM \|r^{0,I}\|_n$
from \eqref{eq:hninit} imply the uniform boundedness of $\epM\|b_3^{0,F}\|_n$, so the modified term $\epA\|b_3^{0,F}\|_n$ for \eqref{eq:hninit}
can be omitted. However, since \eqref{eq:epainvinit} only contains an $O(1)$ estimate for $\pz r^{0,I}$ not $r^{0,I}$ itself, it is not possible
to omit the expression $\Deltainv\Delta_h r^{0,I}$ from the term $\|b_3^{0,F}+\Deltainv\Delta_h r^{0,I}\|_{n}$. Adding \eqref{eq:hninit} and
\eqref{eq:epainvinit} and making the above-mentioned modifications shows that \eqref{eq:initfisbnd} is equivalent to the uniform boundedness of
$\|\vV^0\|_{n}+\fr{\epA}\|(\cLA+\mu\epM)\vV^0\|_{n-1}$ for initial data of the form \eqref{eq:initfis}.
\end{proof}
We note that, by a slight notational exception, the fast part of $b_3^0$ in \eqref{eq:initfis} is
\begin{equation}\label{eq:b30F}
(b_3^0)^F=-\mu(1-\avv)r^{0,S}+\epA b_3^{0,F},
\end{equation}
showing that, to the leading $O(\mu)$ order, $(b_3^0)^F$ only depends on the slow part of $r^0$.

Although Lemma~\ref{lem:initform} determines the most general initial data satisfying the conditions needed for the existence and convergence results, as noted
in the introduction we need more to obtain the rate of convergence. In order to allow the initial data to contain all
modes but not interfere with the convergence rate, we will still assume that the initial data have the
form \eqref{eq:initfis}, and that \eqref{eq:divb00} holds, but will assume that
\begin{equation}\label{eq:initfixed}
  \text{all terms $w^{0,\ell}$ in \eqref{eq:initfis} are uniformly 
    bounded in  $H^n$ and each term $w^{0, S}$  is fixed,}
\end{equation}
which automatically implies that \eqref{eq:initfisbnd} holds.  We allow the $w^{0,F}, w^{0,I}$ to depend on $(\epA,\epM)$
because we only bound the distance of the fast and intermediate modes to zero (i.e., estimate their Sobolev norms), not their distance to any non-trivial limits.

\section{Convergence Rate Estimates}\label{sec:rate}

Recall that ``static'' estimates are obtained by solving the PDE~\eqref{MHDsy} for certain terms and bounding the norms of the result via
\eqref{eq:uniform}, while ``dynamic'' estimates are obtained via energy estimates for the time evolution. 
We will estimate the size of the fast
  modes statically, the size of the intermediate modes statically and then dynamically, and the difference between the
  slow modes and the solution of the limit system dynamically, with earlier estimates 
  used when deriving later ones. To optimize the use of earlier estimates we use the interpolation estimate \eqref{eq:vtinterp} to
  obtain smaller estimates in weaker norms.
  Recall that $\mu=\frac{\epA}{\epM}=\epM^\nu$ is assumed to be less than one. 
 To see easily how the estimate to be obtained depends on the norm used, we  
 introduce an increasing geometric sequence
  \begin{equation}\label{eq:epjdef}
  \ep_j:=\epM^{1+\nu-j\nu}
\end{equation}
{ so that by \eqref{eq:epAepMgen}, \eqref{eq:nudef}, \; $\ep_0=\epA,\;\ep_1=\epM,\;\ep_{j}=\mu\ep_{j+1},\;\ep_n\le c$.}

 \subsection{Static estimates}

  Static estimates will be obtained by solving the PDE system for  $(\cLA +\mu\cLM)\vV$ and using the uniform bounds
  \eqref{eq:uniform} together with the standard Sobolev product and
  composition estimates 
  \begin{align}
    \label{eq:sobprod}
    \|fg\|_j&\le c\|f\|_{n-1}\|g\|_j, \quad j=0,\ldots,n-1
    \\\label{eq:sobfunct}
    \|F(g)\|_{n-1}&\le C(\|g\|_{n-1}),
  \end{align}
  which will all be used henceforth without mention,
  plus the interpolation estimate
  \eqref{eq:vtinterp}. We will also need an estimate for the time integral of certain fast terms, which will be obtained similarly from the
 time integral of the PDE.  
\begin{thm}\label{thm:fast}
  Assume that the basic conditions of Theorem~\ref{thm:theom1} hold. Then
    the fast component $\vV^F$ satisfies the estimates 
\begin{align}
  \label{eq:fastest}
  \begin{aligned}
    \sup_{t\in[0,T]}\big[\big\|\uh^F\big\|_{j}&+ \big\|\nch\uh^F\big\|_{j}
+ {\big\|\pz\vu^F\big\|_j+\big\|\pz\vb^F\big\|_j+\big\| \bh^F\big\|_j+\big\|(1-\az) b_3^F\big\|_j}
\\&{+ \|\az\uh^F\|_{j+1}}+\|b_3^F+\mu \Deltainv\Delta_h  r\|_{j+1}\big]\le 
C\,{\ep_j,\quad j=0\ldots n-1},
\end{aligned}
  \\
\label{TA:epA2}
  \begin{aligned}
    \sup_{t\in[0,T]} \Big[ &\Big\|\int_0^t \az\uh^F\,dt'\Big\|_{n} +  \Big\|\int_0^t \uh^F\,dt'\Big\|_{n-1} +  \Big\|\int_0^t \nch\uh^F\,dt'\Big\|_{n-1}
    \\&
    +  \Big\|\int_0^t \az\left(  b_3^F+\mu \Deltainv\Delta_h  r^S  \right)\,dt'\Big\|_{n}  \Big]
\le C(T) \,\epM^{1+\nu},
\end{aligned}
\end{align}
and the intermediate component $(r^I,u_3^I)=((1-\az)r,(1-\az)u_3)$ satisfies 
\begin{equation}\label{eq:inteststat}
    \|(r^I,u_3^I)\|_{j}\le c\|(\pz r^I,\pz u_3^I)\|_{j}\le c\, \ep_{j+1}=\dfrac{c\,\ep_j}{\mu},\quad j=0\ldots n-1.
\end{equation}
\end{thm}
Note that the bounds in \eqref{TA:epA2} for the time integrals of fast components are smaller than the bounds in \eqref{eq:fastest} for those
components themselves.

\begin{proof}  
Taking the $H^{k}$ norm of both sides of $\epA$ times \eqref{eq:general} and using the 
interpolation bounds \eqref{eq:vtinterp}  to estimate the left side shows that
\begin{equation}
   \label{eq:clamuclmest}
   \|(\cLA+\mu\cLM)\vV\|_{j}\le c\, \epM^{1+\nu-j\nu}=c\,\ep_j. 
 \end{equation}
Combining \eqref{eq:clamuclmest}, 
\eqref{eq:fifromlarge} 
and the Poincar\'e inequality
\begin{equation}
  \label{eq:poincare}
  \|(1-\az)f\|_j\le c \|\pz f\|_j,
\end{equation}
yields \eqref{eq:inteststat}.
By \eqref{eq:fifromlarge}, \eqref{eq:clamuclmest}, and the fact that $u_3^F\equiv0$, 
 \begin{equation}
   \label{eq:b3fest2}
 \|\pz\vu^F\|_j+  \|\nch\uh^F\|_j+\|b_3^F+\mu \Deltainv\Delta_h  r\|_{j+1}\le c\, \ep_j, \quad j=0,\ldots, n-1.
 \end{equation}
Combining  the definition \eqref{eq:modeforms} of the fast modes, the Poincar\'e inequality \eqref{eq:poincare},
the second inequality of \eqref{eq:fifromlarge}, and \eqref{eq:clamuclmest} yields
\begin{equation}
  \label{eq:best1}
  \begin{aligned}
    \big[\|&\pz \vb^F\|_j+\|\bh^F\|_j+\|(1-\az)b_3^F\|_j\big]
    \le c\|\pz\vb^F\|_j 
    \\&\le c\big[ \|\pz(\bh^F-\mu\Deltainv\nabh\pz r^I)\|_j+\mu\|\pz r^I\|_j+\|\pz(b_3^F+\mu\Deltainv\Delta_h r)\|_j+\mu\|\pz r^I\|_j\big]
\\&\le c\|(\cLA+\mu\cLM)\vV\|_j\le c\,\ep_j, \qquad\qquad j=0,\ldots,n-1.
\end{aligned}
\end{equation}
By the definition \eqref{eq:modeforms} of the fast modes 
\begin{equation}
  \label{eq:uhf}
    \uh^F=(1-\az)\uh^F+\az\uh^F=(1-\az)\uh^F +\nabh\Delta_h^{-1}\az(\nch\uh^F).
  \end{equation}
Note that $\Delta_h$ is elliptic when applied to functions independent of $z$ so
\begin{equation}
  \label{eq:deltaellip}
  \|\nabh\Delta_h^{-1}\az f\|_{j+1}\le c\|f\|_j.
\end{equation}
By \eqref{eq:uhf}, \eqref{eq:poincare}, \eqref{eq:deltaellip},
\eqref{eq:fifromlarge}, and \eqref{eq:clamuclmest},
\begin{equation}
  \label{eq:uhfest}
  \begin{aligned}
  \|\uh^F&\|_{j}+\|\az \uh^F\|_{j+1}\le \|(1-\az)\uh^F\|_{j} +\|\nabh\Delta_h^{-1}\az(\nch\uh^F)\|_{j+1}
  \\&\le c\|\pz\uh^F\|_{j}+c\|\nch\uh^F\|_{j}\le c \|(\cLA+\mu\cLM)\vV\|_{j}
  \le c\, \ep_j \quad j=0,\ldots,n-1.
\end{aligned}
\end{equation}
Combining \eqref{eq:b3fest2}, \eqref{eq:best1}, and \eqref{eq:uhfest} yields \eqref{eq:fastest}.

The bounds  on time  integrals in \eqref{TA:epA2} are obtained by integrating \eqref{eq:general}
with respect to time, which yields
\begin{equation}
  \label{eq:intclAetc}
  \int_0^t (\cLA+\mu\cLM)\vV=\epA\left( A_0(\epM\vV)\vV|_0^t -\int_0^t (\epM\vV_t\cdot\grad_{\vV} A_0)\vV+\int_0^t \sum_{i=1}^d A_i(\vV)\,\vV_{x_i}\right).
\end{equation}
For the MHD system \eqref{MHDsy}, only the variable $r$ appears in  the argument of $A_0$, and \eqref{MHDsy:r} shows that the time derivative of $r$
is $O(\epM^{-1})$, so the term $\epM\vV_t\cdot\grad_{\vV} A_0$ is uniformly bounded. This yields the estimate
\begin{equation}
  \label{eq:intclAetc2}
  \Big\|\int_0^t (\cLA+\mu\cLM)\vV\,dt'\Big\|_{n-1}\le c\,\epA=c\,\epM^{1+\nu}.
\end{equation}
Since spatial operators commute with time integration, replacing every solution component in the second inequality of \eqref{eq:fifromlarge} with its
time-integral from $0$ to $t$ and combining the result with the bound \eqref{eq:intclAetc2} yields \eqref{TA:epA2} since $\az r=r^S$.
\end{proof}

\subsection{Intermediate system dynamic estimates}

\begin{thm}\label{thm:iest}
    Assume that the conditions of the convergence part of Theorem~\ref{thm:theom1} hold, and let $(r,\vu,\vb)$ be the solution obtained for the
  MHD system~\eqref{MHDsy}.
  Then
  \begin{equation}
    \label{eq:intestdyn}
    \|r^I\|_j+\|u_3^I\|_j\le c\, \ep_{\max(j,1)}, \quad j=0,\ldots, n-1.
  \end{equation}
\end{thm}
The case $j=0$ in \eqref{eq:intestdyn}
was already proven in Theorem~\ref{thm:fast}.  The remaining cases in \eqref{eq:intestdyn} are an improvement over
the corresponding cases in \eqref{eq:inteststat} by one factor of $\mu$. 

To prove \eqref{eq:intestdyn} we will use the variables
\begin{equation}\label{eq:abdef}
  \begin{aligned}
  \alpha&\eqdef  (1-\az) \vValpnodiv\cdot\vV= (1-\az)(r-\mu b_3+\mu^3 \mathcal{C}\mathcal{Q}\Deltainv\pz^2 b_3),
\\
  \beta&\eqdef (1-\az)\vV_\beta\cdot\vV=(1-\az)(\mathcal{D} u_3+\mu^2 \mathcal{C}\Deltainv \pz \nch\uh),
\end{aligned}
\end{equation}
where the operators $\vValpnodiv$, $\vV_\beta$, $\mathcal{C}$, $\mathcal{D}$, and $\mathcal{Q}$ were defined in \eqref{eq:vabdef}\textendash\eqref{eq:abphipsi}.
As a preliminary 
we will derive a system of PDEs
satisfied by $(\alpha,\beta)$, with remainder terms that are consistent with the desired estimate \eqref{eq:intestdyn}.
The general idea is to apply each of the operators $\vValpnodiv$ and $\vV_\beta$ to the PDE, note that by Lemma~\ref{lem:exacti} the large terms of
the result are $\pz\mathcal{Q}$ applied to the other operator, and calculate the form of the order one terms. To simplify that calculation we first
simplify the original equations by moving to the right sides all terms whose $H^j$ norms can be estimated
by a constant times $\ep_j$  or the $H^j$ norms of $r^I$ and $u_3^I$ using \eqref{eq:fastest}, \eqref{eq:inteststat}, and \eqref{eq:vtinterp}.
To do so we will the formulas~\eqref{eq:modeforms} and in particular their consequence
\begin{equation}\label{eq:vucnform}
  \vu\cn=(\vu^S+\vu^I+\vu^F)\cn=\vu^S\cn+u_3^I\pz+\uh^F\cnh.
\end{equation}

Starting from the MHD equations \eqref{MHDsy}, 
we replace the argument $\epM r$ of $\rho$ and $a$ by $\epM r^S$ or zero where possible and compensate by adding terms to the right sides of the
equations, except that we retain a factor of $a(\epM r^S)$ multiplying
$\tfrac1{\epM}\nch\uh$ because that expression will then cancel exactly when we build the time evolution equation for $\alpha$, \eqref{eq:alpqnew}.
In addition, we apply $1-\az$ to the equations since that operator appears in all terms of the formulas \eqref{eq:abdef} for $\alpha$ and $\beta$.
Since slow modes are independent of $z$ the operator $1-\az$ can be moved past most coefficients. 
This  yields
\begin{subequations}\label{eq:MHDsymod}
  \begin{align}
    \label{eq:MHDsymod:r}
    &a(\epM r^S) (\pt +(\vu^S\cn))r^I+\tfrac{a(\epM r^S)\rho(\epM r^S)}{\epM}\pz u_3^I+\tfrac{a(\epM r^S)}{\epM} (1-\az)\nch\uh=(1-\az)R_1,
    \\
    \label{eq:MHDsymod:uh}
    &\begin{aligned}\rho(\epM &r^S)(\pt +(\vu^S\cn))(1-\az)\uh+\tfrac{a(\epM r^S)\rho(\epM r^S)-1}{\epM}\nabh r^I
      \\&+(1-\az)[\nabh\tfrac{|\vb|^2}2-(\vb\cn)\bh]
            -\tfrac{(1-\az)}\epA(\pz \bh^F-\nabh b_3^F-\mu\nabh r^I)=(1-\az)R_2,
  \end{aligned}
    \\
    \label{eq:MHDsymod:u3}
    &\begin{aligned}\rho(\epM &r^S)(\pt +(\vu^S\cn) )u_3^I+\tfrac{a(\epM r^S)\rho(\epM r^S)}{\epM}\pz r^I-(\bh^S\cnh)(1-\az)b_3^F
       \\&=(1-\az)R_3,
    \end{aligned}
    \\
    \label{eq:MHDsymod:b3}
    &(\pt +(\vu^S\cn))(1-\az)b_3-(\bh^S\cnh)u_3^I+(1-\az)\tfrac1\epA\nch\uh=(1-\az)R_4,
  \end{align}
\end{subequations}
where the equation for $\bh$ has been omitted since it does not enter into $\alpha$ or $\beta$, and
\begin{subequations}
  \begin{align}
    \label{eq:R1def}
    &\begin{aligned}
      R_1&\eqdef-a(\epM r)\left( u_3^I\pz r+(\uh^F\cnh)r\right)+\left( a(\epM r^S)-a(\epM r)\right) (\pt r+(\vu^S\cn)r)
      \\&\quad\! +\frac{a(\epM r^S)\rho(\epM r^S)-a(\epM r)\rho(\epM r)}{\epM}(\nch\uh+\pz u_3^I)
      +a(\epM r^S)\frac{1-\rho(\epM r^S)}{\epM} \nch\uh,
  \end{aligned}
    \\
    \label{eq:R2def}
    &\begin{aligned}
      R_2&\eqdef-\rho(\epM r)\left( u_3^I\pz \uh+(\uh^F\cnh)\uh\right)+\left( \rho(\epM r^S)-\rho(\epM r)\right) (\pt \uh+(\vu^S\cn)\uh)
      \\&\quad +\frac{a(\epM r^S)\rho(\epM r^S)-a(\epM r)\rho(\epM r)}{\epM}\nabh r,
    \end{aligned}
    \\
    \label{eq:R3def}
    &\begin{aligned}
      R_3&\eqdef-\rho(\epM r)\left( u_3^I\pz u_3+(\uh^F\cnh) u_3\right)+\left( \rho(\epM r^S)-\rho(\epM r)\right) (\pt u_3+(\vu^S\cn)u_3)
      \\&\quad +\frac{a(\epM r^S)\rho(\epM r^S)-a(\epM r)\rho(\epM r)}{\epM}\pz  r
      -\bh\cdot \pz \bh^F+(\bh^F\cnh)b_3^F,
    \end{aligned}
    \\
    \label{eq:R4def}
    &R_4=-\left( u_3^I\pz b_3+(\uh^F\cnh) b_3\right)-(\nch\uh)b_3+(\bh^F\cnh)u_3.
  \end{align}
\end{subequations}
Wherever $u_3^I$ appears undifferentiated in $R_i$, the $H^j$ norm of the term in which it appears can be estimated
by a constant times the $H^j$ norm of $u_3^I$, for $j=0,\ldots,n-1$. Similarly, for any smooth function $F$ and $j=0,\ldots,n-1$, 
$\|\tfrac{F(\epM r^S)-F(\epM r)}{\epM}\|_j\le c\|r^I\|_j$, and
$\|[F(\epM r^S)-F(\epM r)](\pt \vV+(\vu^S\cn)\vV)\|_j
  \le c \epM \|r^I\|_j (\|\vV_t\|_{n-1}+\|\vV\|_n)\le c \|r^I\|_j$.
By \eqref{eq:fastest}, terms containing $\uh^F$, $\bh^F$, $\nch\uh^F=\nch\uh$,  or $\pz\bh^F$
without further derivatives can be estimated in the $H^j$ norm by $c\ep_j$, for $0\le j\le n-1$.
Since these cases cover all the terms in the $R_i$,
\begin{equation}
  \label{eq:Rest}
  \sum_{i=1}^4 \|(1-\az)R_i\|_j\le\sum_{i=1}^4 \|R_i\|_j\le c(\|(r^I,u_3^I)\|_j+\ep_j), \quad j=0,\ldots,n-1.
\end{equation}

To obtain the evolution equation for $\alpha$,
subtract $\mu a(\epM r^S)$ times \eqref{eq:MHDsymod:b3} from \eqref{eq:MHDsymod:r} and add
$a(\epM r^S)\mu^3\mathcal{C}\mathcal{Q}\Deltainv\pz^2$ applied to
\eqref{eq:MHDsymod:b3} to the result.
Then commute  $\mathcal{C}\mathcal{Q}\Deltainv\pz^2$ past $\vu^S\cn$, make the
coefficient of the large terms that do not cancel be $a\rho$ everywhere while compensating via terms on the right side,
force the function to which $\mu(\bh^S\cnh)$ is applied to be $\frac{\beta}{\sqrt{1+\mu^2}}$
for reasons to be explained later and again compensate on the right side,
and use the identity $1=\mathcal{Q}\mathcal{D}$ from \eqref{eq:abphipsi} that makes the large terms exactly involve $\beta$,
as we know from \eqref{eq:vabq} that they must. This 
yields
\begin{equation}
  \label{eq:alpqnew}
  \begin{aligned}
    a(\epM r^S)&(\pt+(\vu^S\cn))\alpha +\tfrac{\mu}{\sqrt{1+\mu^2}}(\bh^S\cnh)\beta+\frac{a(\epM r^S)\rho(\epM r^S)}{\epM}\mathcal{Q}\pz\beta
    =R_{\alpha,1}+R_{\alpha,2}+R_{\alpha,3}+R_{\alpha,4},
\end{aligned}
\end{equation}
where
\begin{align}
  \label{eq:Ralp1}
  &R_{\alpha,1}\eqdef(1-\az)R_1-\mu a(\epM r^S)(1-\az)R_4+\mu^3a(\epM r^S)\mathcal{C}\mathcal{Q}\Deltainv\pz^2 R_4
\\\intertext{comes from the right sides of the modified equations \eqref{eq:MHDsymod},}
\label{eq:Ralph2}
  &R_{\alpha,2}\eqdef -\mu^3 a(\epM r^S)[\mathcal{C}\mathcal{Q}\Deltainv\pz^2,\vu^S]\cn(1-\az)b_3^F
\\\intertext{comes from commuting the operator applied to \eqref{eq:MHDsymod:b3} past the coefficient $u_3^S$,}
  \label{eq:Ralp3}
  &\begin{aligned}
R_{\alpha,3}&\eqdef 
-\mu\epM \tfrac{a(\epM r^S)-1}{\epM} (\bh^S\cnh)u_3^I+
\mu^3a(\epM r^S)\mathcal{C}\mathcal{Q}\Deltainv\pz\left[(\bh^S\cnh)\pz u_3^I\right]
\\&\qquad+\mu^3 a(\epM r^S)\tfrac{\rho(\epM r^S)-1}{\epA}\mathcal{C}\mathcal{Q}\Deltainv\pz^2(\nch\uh)
\end{aligned}
  \\\intertext{comes from adding compensating terms to the right side and moving  entire terms there, and}
  \label{eq:ra4def}
&R_{\alpha,4}\eqdef \mu(\bh^S\cnh)\big( \tfrac{\beta}{\sqrt{1+\mu^2}}-u_3^I\big)
\end{align}
comes from forcing the term involving $\bh^S\cnh$ to have the desired form, and will be rearranged further later.

Similarly, to obtain the evolution equation for $\beta$, 
add $\mu^2\mathcal{C}\Deltainv\pz\nch$ applied to \eqref{eq:MHDsymod:uh} to $\mathcal{D}$
applied to \eqref{eq:MHDsymod:u3}, and rearrange terms in
similar fashion as for \eqref{eq:alpqnew}. Then force the function to which $\bh^S\cnh$ is applied on the left side of the equation to be
$\frac{\mu}{\sqrt{1+\mu^2}}\alpha$ for reasons to be explained later, and
subtract appropriate constants from the factors appearing inside commutators since that does not affect their value, in order to
facilitate estimate the size of the resulting terms. Also, 
to ensure the symmetry of the resulting system for $\alpha$ and $\beta$, multiply the large term
$\tfrac{\mu^2}{\epA}(1-\az)\mathcal{C}\Deltainv\pz\nch (\pz\bh-\nabh b_3-\mu\nabh r)$ appearing on the left side of the equation
by $a(\epM r^S)\rho(\epM r^S)$ and compensate by adding $a(\epM r^S)\rho(\epM r^S)-1$ times that term to the right side.
By \eqref{eq:abphipsi}, the resulting large term exactly involves $\alpha$.
This yields
\begin{equation}
  \label{eq:betqnew}
  \begin{aligned}
    \rho(\epM r^S)&(\pt+(\vu^S\cn))\beta+\tfrac{\mu}{\sqrt{1+\mu^2}}(\bh^S\cnh)\alpha+\frac{a(\epM r^S)\rho(\epM r^S)}{\epM}\mathcal{Q}\pz\alpha
    =R_{\beta,1}+R_{\beta,2}+R_{\beta,3}+R_{\beta,4},
  \end{aligned}
\end{equation}
where
\begin{equation*}
  R_{\beta,1}\eqdef\mathcal{D}(1-\az)R_3+\mu^2\mathcal{C}\Deltainv\pz\nch(1-\az)R_2
\end{equation*}
comes from the right sides of \eqref{eq:MHDsymod},
\begin{equation*}
  \begin{aligned}
    R_{\beta,2}&=-\mu^2 \epM[\tfrac{\mathcal{D}-1}{\mu^2},\tfrac{\rho(\epM r^S)-1}{\epM}](\pt +(\vu^S\cn))u_3^I-\mu^2\rho(\epM r^S)[\tfrac{\mathcal{D}-1}{\mu^2},\vu^S]\cn u_3^I 
    -\mu^2[\tfrac{\mathcal{D}-1}{\mu^2},\tfrac{a(\epM r^S)\rho(\epM r^S)-1}{\epM}]\pz r^I
     \\&\quad-\mu^2\epM[\mathcal{C}\Deltainv\pz\nabh,\tfrac{\rho(\epM r^S)-1}{\epM}]\cdot(\pt+(\vu^S\cn))(1-\az)\uh^F
     -\mu^2\rho(\epM r^S)[\mathcal{C}\Deltainv\pz\nabh,\vu^S]\nab(1-\az)\uh^F
     \\&\quad-\mu^2[\mathcal{C}\Deltainv\pz\nabh,\tfrac{a(\epM r^S)\rho(\epM r^S)-1}{\epM}]\cdot \nabh r^I
       \end{aligned}
     \end{equation*}
comes from commuting the operators applied to \eqref{eq:MHDsymod:u3} and \eqref{eq:MHDsymod:uh} past coefficients in those equations,
\begin{equation*}
  \begin{aligned}
    R_{\beta,3}&\eqdef
    -\mu^2\tfrac{a(\epM r^S)\rho(\epM r^S)-1}{\epM}\mathcal{C}\Deltainv\Delta_h\pz r^I
    {-\mu^2\mathcal{C}\Deltainv\pz\nabh\cdot(\nabh\tfrac{|\vb|^2}2-(\vb\cn)\bh)}
    \\&\quad
    -\mu^2\tfrac{a(\epM r^S)\rho(\epM r^S)-1}{\epA}\mathcal{C}\pz (b_3^F+\mu\Deltainv\Delta_h r)
  \end{aligned}
\end{equation*}
comes from moving terms to the right side,  balancing a term added on the left side, and using \eqref{eq:getlapb3}, and
\begin{equation}
\label{eq:rbet4def}
  R_{\beta,4}\eqdef \tfrac{\mu}{\sqrt{1+\mu^2}}(\bh^S\cnh)\alpha+\mathcal{D}(\bh^S\cnh)(1-\az)b_3^F
\end{equation}
comes from forcing the term involving $\bh^S\cnh$ to have the desired form, and will be rearranged further later. 

We now estimate the terms $R_{\alpha,i}$ and $R_{\beta,i}$. Since the operators applied to the $R_i$ in $R_{\alpha,1}$ and $R_{\beta,1}$ are all
bounded, \eqref{eq:Rest} implies that
\begin{equation}
  \label{eq:rab1est}
  \|R_{\alpha,1}\|_j+\|R_{\beta,1}\|_j\le c(\|(r^I,u_3^I)\|_j+\ep_j), \quad j=0,\ldots,n-1.
\end{equation}

The terms in $R_{\alpha,2}$ and $R_{\beta,2}$ all involve commutators, and the
following lemma says that the commutator gains one derivative, which in many cases is a vital improvement.
  \begin{lem}\label{lem:pdcomm}  {\rm(\cite[Lemma 2.5]{metivier01:euler})}
    Let $\widehat P(k,l,m)$ be homogeneous of degree zero and real analytic for $(k,l,m)\ne(0,0,0)$, and let $P$ be the Fourier multiplier
    operator defined by $\widehat{Pf}=\widehat{P}\widehat{f}$.
        Then for all $n\ge3$, $f\in H^n(\mT^3)$, $j\in 1,\ldots,n$, and $g\in H^{j-1}$,
  \begin{equation}
    \label{eq:psdcommest}
    \|[P,f]g\|_j\le c\|f\|_n\|g\|_{j-1}.
  \end{equation}
\end{lem}
The constant-coefficient pseudo-differential operators appearing in the commutators in $R_{\alpha,2}$ and $R_{\beta,2}$ are all homogeneous of degree
zero and bounded uniformly in $\mu$ (in particular \eqref{eq:dformnew} below implies a bound on $\tfrac{\mathcal{D}-1}{\mu^2}$),  and  they are real analytic for
$(k,l,m)\ne(0,0,0)$ since the denominators in the formulas for $\mathcal{C}$, $\mathcal{D}$, and $\mathcal{Q}$  in \eqref{eq:qdef},
\eqref{eq:abphipsi}  are positive for $\mu<1$, so they satisfy the conditions of ~Lemma~\ref{lem:pdcomm}. The expressions to which the commutators are
applied have one of two forms: either they consist of a single spatial derivative applied to a fast or intermediate component that is estimated in
\eqref{eq:fastest} or \eqref{eq:inteststat}, or they are some component of $(\pt+(\vu^S\cn))\vV$. In the former case the expression contains a factor
$\mu^p$ with $p\ge2$, so by Lemma~\ref{lem:pdcomm} its $H^j$ norm for $j=0,\ldots,n-1$ can be estimated by
\begin{equation*}
  \begin{aligned}
  c\mu^2 \|\nab ((1-\az)b_3^F,u_3^I, r^I,\uh^F)\|_{\max(0,j-1)}&\le c\mu^2 \|((1-\az)b_3^F,u_3^I, r^I,\uh^F)\|_{\max(1,j)}
  \le c\mu^2 \ep_{j+2}\le c\,\ep_j.
\end{aligned}
\end{equation*}
In the latter case the expression contains the factor $\mu^2\epM$
so by the interpolation bounds \eqref{eq:vtinterp} its $H^j$  norm for $j=0,\ldots,n-1$ can be estimated by
$c\mu^2\epM(\|\vV_t\|_{j-1}+c)\le c\mu^2\epM \mu^{-j}\le c\mu\ep_j$. Hence
\begin{equation}
  \label{eq:rab2est}
   \|R_{\alpha,2}\|_j+\|R_{\beta,2}\|_j\le c\, \ep_j, \quad j=0,\ldots,n-1.
\end{equation}

To estimate the terms in $R_{\alpha,3}$ and $R_{\beta,3}$ note that those terms either have a factor $\mu^2\epM$ that is smaller than all $\ep_j$,
have a zeroth-order pseudo-differential operator applied to $\pz r^I$, $\vb_z$ (after applying the $z$-derivative in the
expression to some factor of $\vb$), $\nch\uh$, or
$\pz(b_3^F+\mu\Deltainv\Delta_h r)$, all of which have their $H^j$ norms estimated in \eqref{eq:fastest} or \eqref{eq:inteststat}, or have a
zeroth-order pseudo-differential operator applied to $\Deltainv\pz[(\bh\cnh)\pz u_3^I]$. In the first case using the uniform estimate
\eqref{eq:uniform} shows that the $H^{n-1}$ norm is bounded by $c\,\ep_0$. In the middle cases each term contains a factor of $\mu^2$
so its $H^j$ norm is bounded by $c\mu^2 \ep_{j+1}\le c\,\ep_j$. In the final case, since $\|\Deltainv\pz f\|_j\le c\|f\|_{\max(j-1,0)}$ and the term
contains a factor of $\mu^3$, it is bounded by $c\mu^3 \|\nabla\pz u_3^I\|_{\max(j-1,0)}\le c \mu^3\|\pz u_3^I\|_{\max(j,1)}\le
c\mu^3\ep_{j+2}\le c\mu\ep_j$. Together, these yield
\begin{equation}
  \label{eq:rab3est}
   \|R_{\alpha,3}\|_j+\|R_{\beta,3}\|_j\le c\, \ep_j, \quad j=0,\ldots,n-1.
\end{equation}

For $j\le n-2$ the terms $R_{\alpha,4}$ and $R_{\beta,4}$ can be estimated by using the fact that
\eqref{eq:fastest} and \eqref{eq:inteststat} imply that 
\begin{equation}
  \label{eq:estjlen-2}
  \mu\|\vV^F_*\|_{j+1}+\mu^2\|\vV^I\|_{j+1}\le c(\mu\ep_{j+1}+\mu^2\ep_{j+2})\le c\,\ep_j, \quad j=0,\ldots, n-2,
\end{equation}
where $\vV^F_*$ means all components of $\vV^F$ except $\az b_3^F$, which is not estimated in \eqref{eq:fastest}.
The estimate $\|R_{\alpha,4}\|_j+\|R_{\beta,4}\|_j\le c \ep_j$ can therefore be obtained for $j\le n-2$ by using the formulas for $\alpha$ and $\beta$
in \eqref{eq:abdef} together with the facts that $\frac{\mathcal{D}-1}{\mu^2}$ is a bounded zeroth-order pseudo-differential
operator and that 
\begin{equation}\label{eq:b3fform}
  b_3^F=(b_3^F+\mu\Deltainv\Delta_h r)+\mu\Deltainv\pz^2 r^I-\mu (r-\avv r)
\end{equation}
plus estimates similar to those used for $R_{\alpha,3}$ and $R_{\beta,3}$.

However, \eqref{eq:estjlen-2} is not valid for $j=n-1$ because \eqref{eq:fastest} and \eqref{eq:inteststat} do not hold for $j=n$.
We therefore rearrange $R_{\alpha,4}$ and $R_{\beta,4}$ to be linear
combinations of the terms
\begin{equation}\label{eq:goodterms}
  \nch\uh, \quad \pz \vu^F, \quad \pz \vb, \quad\grad(b_3^F+\Deltainv\Delta_h r), \quad \pz r \quad\text{and}\quad \pz u_3
\end{equation}
whose $H^j$ norms are estimated
in \eqref{eq:fastest} or \eqref{eq:inteststat} even though they involve a first derivative of $\vV$; in addition a factor of $\mu$ must be present
multiplying the terms $\pz(r^I,u_3^I)$ to compensate for the extra factor of $\tfrac1\mu$ in \eqref{eq:inteststat}.
Substituting \eqref{eq:b3fform} into the formula for $\alpha$ in \eqref{eq:abdef} and solving the result for $r^I$ yields
\begin{equation}
  \label{eq:riform}
  (1+\mu^2)r^I=\alpha+\mu(1-\az)(b_3^F+\mu\Deltainv\Delta_h r)+\mu^2 \Deltainv\pz^2 r^I-\mu^3\mathcal{C}\mathcal{Q}\Deltainv\pz^2 r^I.
\end{equation}
Applying $(1-\az)$ to \eqref{eq:b3fform}, which turns the final $-\mu (r-\avv r)$ into $-\mu r^I$, dividing \eqref{eq:riform} by $1+\mu^2$ and substituting the
result for that final $r^I$, and substituting the result into \eqref{eq:rbet4def} shows that $R_{\beta,4}$ equals
\begin{equation}\label{eq:mudform}
  \mu(\tfrac1{\sqrt{1+\mu^2}}-\tfrac1{1+\mu^2}\mathcal{D})(\bh\cnh)\alpha
\end{equation}
plus a sum of terms involving operators of order zero applied to the expressions in \eqref{eq:goodterms} whose $H^j$ norms can be bounded by $c\,\ep_j$ using
\eqref{eq:fastest} and \eqref{eq:inteststat}. 
To estimate \eqref{eq:mudform} we use the identity
\begin{equation}
  \label{eq:dformnew}
  \begin{aligned}
  \mathcal{D}&=\sqrt{1+\mu^2}
  -\tfrac{4 \mu^2\Deltainv\pz^2}{\left(\sqrt{\mu ^2 \left(2-4\Deltainv\pz^2\right)+\mu ^4+1}+\mu ^2+1\right) \left(\sqrt{2}
   \sqrt{\sqrt{\mu ^2 \left(2-4\Deltainv\pz^2\right)+\mu ^4+1}+\mu
   ^2+1}+2 \sqrt{\mu ^2+1}\right)}
\end{aligned}
\end{equation}
derived from \eqref{eq:abphipsi}, which makes the constant term in \eqref{eq:mudform} cancel. The other term in \eqref{eq:dformnew} contains a factor of $\pz$
multiplied by an operator of order $-1$, the $z$-derivative of all the constituents of $\alpha$ are estimated in \eqref{eq:fastest} and \eqref{eq:inteststat}, and
an overall factor of $\mu$ is present in \eqref{eq:dformnew}, so
the $H^j$ norm of $R_{\beta,4}$ is bounded by a constant times $\ep_j$ even for $j=n-1$. Similarly, substituting formula \eqref{eq:abdef}
for $\beta$ into the definition \eqref{eq:ra4def} of $R_{\alpha,4}$ and substituting \eqref{eq:dformnew} into the result makes the constant term from
\eqref{eq:dformnew} cancel. All
remaining terms involve the expressions from \eqref{eq:goodterms} and so can be estimated by a constant times $\ep_j$ since an overall factor
of $\mu$ is present. This yields
\begin{equation}
  \label{eq:rab4est}
   \|R_{\alpha,4}\|_j+\|R_{\beta,4}\|_j\le c\, \ep_j, \quad j=0,\ldots,n-1.
\end{equation}

\begin{proof}[Proof of Theorem~\ref{thm:iest}]
  The system \eqref{eq:alpqnew}, \eqref{eq:betqnew} has the form of the Klainerman-Majda two-scale theory.
  Moreover, \eqref{eq:abdef} plus estimates similar to those above yield
  \begin{equation}
    \label{eq:alpbetriu3i}
    \|\alpha-r^I\|_j+\|\beta-u_3^I\|_j\le c\,\ep_j,\quad j=0,\ldots,n-1.
  \end{equation}
Together with  the estimates
  \eqref{eq:rab1est}, \eqref{eq:rab2est}, \eqref{eq:rab3est}, and \eqref{eq:rab4est}, \eqref{eq:alpbetriu3i}  shows that the $H^j$ norm of the right sides of those equations is
  bounded by a constant times $\|(\alpha,\beta)\|_j+\ep_j$. Hence the standard Klainerman-Majda energy estimates show that
  \begin{equation}
    \label{eq:alpbetest}
   \max_{0\le t\le T} \|(\alpha,\beta)\|_j\le c(\|\alpha(0),\beta(0))\|_j+\ep_j), \quad j=0,\ldots,n-1.
 \end{equation}
 The initial data \eqref{eq:initfis},  \eqref{eq:initfixed} satisfies   $\|\vV^I\|_{n-1}\big|_{t=0}\le c\,\epM=c\,\ep_1$, so
\eqref{eq:alpbetriu3i} implies that $\|\alpha(0),\beta(0))\|_j\le c\ep_1$ and hence
 \eqref{eq:alpbetest} implies that $\max_{0\le t\le T} \|(\alpha,\beta)\|_j\le c\, \max(\ep_j,\ep_1)$ for $j=0,\ldots,n-1$.  Using \eqref{eq:alpbetriu3i} once more
 yields \eqref{eq:intestdyn}.
\end{proof}

\subsection{Equations and estimates for horizontal components of the slow mode}

Like for the intermediate mode dynamic estimates, estimating the difference between the slow modes of the solution to the
original system and the solution of the limit system requires PDEs for the exact zero eigenspace of the operator
$\cLA+\mu\cLM$.  The horizontal velocity and magnetic field slow modes belong to that eigenspace, so we will
write the equations for those modes as the limit equations plus error terms, by applying the projection
$\mathbb{P}$ onto the slow horizontal modes to the original system, expanding all dependent variables into fast, intermediate, and slow modes, and moving
all terms except those involving purely slow modes to the right sides of the equations. The remaining slow modes will be
treated in the following subsection.

\begin{thm}\label{thm:slow1}
  Assume that the conditions of the convergence part of Theorem~\ref{thm:theom1} hold.  Let $(r,\vu,\vb)$ be the solution of the MHD system~\eqref{MHDsy}, and
 let $(\rlim, \uhlim, \uvlim, \bhlim,\bvlim)$ be the solution of the limit system \eqref{limit:CJS0}  whose 
  initial data is the limit $(r^{0,S},\vu^{0,S}, \vb^{0,S}+\mu_{\lim}(1-\avv)r^{0,S})$ of the initial data \eqref{eq:initfis} of
  the original system.  Then
\begin{align}
  \label{eq:slowest1}
  \|\uh^S-\uhlim\|_{n-2}+\|\bh^S-\bhlim\|_{n-2}\le c\,\epM.   
\end{align}
\end{thm}
Before proving Theorem~\ref{thm:slow1} we need to derive appropriate equations. Since all slow modes contain the averaging operator $\az$ in
the $z$ direction, it will be convenient to write the equations in conservation form, so that derivatives with respect to $z$ disappear
when $\az$ is applied. In particular,
\begin{equation}\label{eq:divf}
  \az(\vw\cn f + (\nc\vw) f) =\az \nc(f\vw)=\az\nch(f\wh)=\az(\wh\cnh f +(\nch\wh) f).
\end{equation}
Adding the continuity equation \eqref{rho:0} to the momentum equation
\eqref{MHDsy:vu} yields
\begin{equation}\label{eq:ucons}
\pt(\rho\vu)+ \vu\cn(\rho\vu)+(\nc\vu)(\rho\vu) +  \nabla {\Phi}-\vb\cn\vb
=\epA^{-1} (\pz\vb-\nabla b_3),
\end{equation}
where ${\Phi}$ is some scalar-valued function defined on $\mT^3$,  i.e., periodic 
in $x,y,z$.
Apply $\cPdh\az$ to the first two components of \eqref{eq:ucons} and $\az$ to the first two components of \eqref{MHDsy:vb}, and simplify the result
by using the identity \eqref{eq:divf} not only with $\vw=\vu$ but also with $\vw=\vb$ since the constraint $\nc\vb=0$ implies that $\vb\cn f$ equals
$\nc(f\vb)$. The resulting equations can be further simplified by using the definitions \eqref{eq:modeforms} of the modes to obtain the identities
$\az\bh=\bh^S$ and $\az(\rho\uh)=\az(\rho\uh^S)+\az(\rho\uh^F)=(\az\rho)\uh^S+\az(\rho\uh^F)$, and then using the facts that $\rho=1+\epM r$ and
$\az\uh^F$ is a gradient to obtain $\cPdh\az(\rho\uh^F)=\epM\cPdh\az(r\uh^F)$.  This yields \begin{subequations}\label{slow:dyn}
\begin{align}
\label{slow:vu}\pt\,\cPdh((\az\rho)\uh^S)&+\cPdh\az\nch\big(\rho\, \uh\otimes\uh-\bh\otimes\bh\big)=-\epM\,\pt\,\cPdh\az(r\uh^F),\\
\label{slow:vb}\pt\,\bh^S&+\phantom{\cPdh}\az\nch\big(\uh\otimes\bh-\bh\otimes\uh\big)\;\;\,=0  ,\end{align}
\end{subequations}
where the tensor product $\otimes$ follows the convention that
$\nch(\uh\otimes\bh)=\uh\cnh\bh +(\nch\uh)\bh$.
Since $\pt\cPdh((\az\rho)\uh^S)=\pt\uh^S +\epM \pt \cPdh((\az r)\uh^S)$, \eqref{slow:dyn} together with the bounds
\eqref{eq:uniform} and \eqref{eq:epAepMgen} and the relation~\eqref{eq:nudef} show that
\begin{equation}
  \label{eq:ptuhsbhs}
  \|\pt \uh^S\|_{n-1}+\|\pt\bh^S\|_{n-1}\le c.
\end{equation}

Recalling that $\uh$ and $\bh$ have no intermediate part, we separate them into their fast
and slow parts  in the tensor products in \eqref{slow:vu}:
\begin{equation}\label{eq:sepsf}
\begin{aligned}
\az\nch\big(\rho\, \uh\otimes\uh-\bh\otimes\bh\big) &=\,\az\nch\big(\rho\,
\uh^S\otimes\uh^S-\bh^S\otimes\bh^S\big)\,  
\\
&\quad+\az\nch\,\Big(\big(\rho\, \uh^S\otimes\uh^F-\bh^S\otimes\bh^F\big)+\trsp\Big)+\az\nch\,\big(\rho \,\uh^F\otimes\uh^F - \bh^F\otimes\bh^F\big)\,
\end{aligned}
\end{equation}
where $\trsp$ denotes the tensor transpose. By
\eqref{eq:ind:z}--\eqref{eq:h:div:0}, the slow parts $(\uh^S,\bh^S)$
are independent of $z$ and divergence-free, so the first term on the right
side of~\eqref{eq:sepsf} simplifies to
$\uh^S\cnh((\az\rho)\uh^S)-\bh^S\cnh\bh^S$.  Also, since $\bh^S$ is
  independent of~$z$ while \eqref{eq:modeforms} shows that
  $\az\bh^F=0$, the expression $-\az\nch
  (\bh^S\otimes\bh^F)$ appearing in  the second term on the right in \eqref{eq:sepsf} is identically zero.
  Next, we  can drop the
$\cPdh$ operator from \eqref{slow:vu}  at the cost of adding a term 
$\nabh\theta(t,x,y)$  to that equation, since  a 2-vector is in the kernel of
$\cPdh$ if and only if it is a horizontal gradient.
Thus, \eqref{slow:vu} becomes
\begin{equation*}
  \begin{aligned}&\pt ((\az\rho)\uh^S)+\uh^S\cnh((\az\rho)\uh^S)-\bh^S\cnh\bh^S\\
    &=
    -\epM\pt\,\az(r\uh^F)
-\az\!\nch\, \big( \rho\,\uh^S\otimes\uh^F +\trsp\big)
-\az\!\nch\big(\rho\, \uh^F\otimes\uh^F- \bh^F\otimes\bh^F\big)+\nabh \theta\,.\end{aligned}
\end{equation*}
 The second term on the right side is a  ``slow-fast'' product, which can be rewritten
 using time-integrated variable
 \begin{equation}
   \label{eq:Adef}
   \vA(t,\cdot):=\int_0^t\az\uh^F\,dt'\,,
 \end{equation}
 as the time derivative of a small term plus a small term, since \eqref{TA:epA2} shows that
 \begin{equation}
   \label{eq:bigaest}
   \|\vA\|_n\le c\,\epM^{1+\nu}.
 \end{equation}
For example, 
\begin{align*}
  \az(\rho\,\uh^S\otimes\uh^F)
&=\az (\uh^S\otimes\uh^F)+\epM \az(r \uh^S\otimes\uh^F)
=\uh^S\otimes (\az\uh^F) +\epM \az(r \uh^S\otimes\uh^F)
  \\&  =\pt(\uh^S\otimes\vA)-\big(\pt\uh^S\big)\otimes\vA +\epM \az(r \uh^S\otimes\uh^F)\,.
\end{align*}
Hence we obtain
\begin{equation}
 \label{slow:u:with:ave} (\pt+(\uh^S\cnh))\big((\az\rho)\uh^S\big) -\bh^S\cnh\bh^S
   =\pt\,\Xi_1+\xi_1+\nabh\theta\,, \qquad \nch\uh^S=0
 \end{equation}
 with
 \begin{equation*}
   \begin{aligned}
    \Xi_1 &:=- \epM \az(r\uh^F)
    -\nch\,\big(\uh^S\otimes\vA+\trsp\big) , \\
    \xi_1 &:= \nch\Big\{\big(\pt \uh^S\big)\otimes\vA+\trsp\Big\}
  -\epM\! \az\!\nch(r \uh^S\otimes\uh^F+\trsp)
  -\az\!\nch\big(\rho\, \uh^F\otimes\uh^F - \bh^F\otimes\bh^F\big)\,.
 \end{aligned}
\end{equation*}
The bound \eqref{eq:fastest} together with the constraint \eqref{eq:epAepMgen},  the
relation~\eqref{eq:nudef} between the parameters and the definition \eqref{eq:epjdef} of the $\ep_j$ implies that $\|\uh^F\|_{n-1}+\|\bh^F\|_{n-1}\le
c\,\epM^\nu$. Using that estimate, the time-derivative bound \eqref{eq:ptuhsbhs}, estimate \eqref{eq:bigaest} for $\vA$, the formula $\rho=1+\epM
r$,  and Corollary~\ref{cor:azproduct} yields the estimates
\begin{equation}
  \label{eq:xi1est}
  \|\Xi_1\|_{n-1}+\|\xi_1\|_{n-2}\le c\,\epM^{1+\nu}.
\end{equation}

Applying the same ideas to the slow magnetic equation \eqref{slow:vb}, and in particular noting that $\az (\uh^S\otimes \bh^F+\trsp )=0$
yields
\begin{equation}
  \label{slow:b:with:ave}
   (\pt+(\uh^S\cnh)) \bh^S-\bh^S\cnh\uh^S=\pt\,\Xi_{\bh^S}+\xi_{\bh^S}\,, \qquad \nch\bh^S=0,
 \end{equation}  
with
\begin{equation}
  \label{eq:xi2}
 \begin{aligned} \Xi_{\bh^S}& := -\nch\,\big[ \vA\otimes\bh^S-\trsp\big]\,, \quad
  {\xi_{\bh^S}}& {:= \nch\,\big[ \vA\otimes\pt\bh^S-\trsp\big]-\az\nch\big( \uh^F\otimes\bh^F - \trsp\big).}
  \end{aligned}
\end{equation}
The same bounds as for \eqref{eq:xi1est} show that 
\begin{equation}
  \label{eq:xibhsest}
  \begin{aligned}
    \|\Xi_{\bh^S}\|_{n-1}+\|\xi_{\bh^S}\|_{n-2}&\le c\,\epM^{1+\nu}.
\end{aligned}
\end{equation}

The equation for the time evolution of $\uh^S$ can be further simplified using an equation for the time evolution of $r^S$.
By the first part
of~\eqref{eq:h:div:0} and the facts from \eqref{eq:modeforms}  that $\uh^I=0$ and
 $\az r=r^S$, the vertically-averaged equation \eqref{eq:rforlim} simplifies 
to
\begin{equation}\label{mod:dyn:az:r}\pt\, r^S+\uh^S\cnh( r^S)=-\az\nch(r\uh^F)-\epM^{-1}\nch(\az\uh^F)\,.
\end{equation}
Using the time-integrated variable~$\vA$ from \eqref{eq:Adef}, \eqref{mod:dyn:az:r} can be rewritten as
\begin{equation}
  \label{eq:ptepMrS}
  (\pt+(\uh^S\cnh))(\epM r^S) =-\epM\az\nch(r\uh^F)-\pt \nch\vA.
\end{equation}
Subtracting $\uh^S$ times \eqref{eq:ptepMrS} from \eqref{slow:u:with:ave}
noting that $\az\rho=1+\epM r^S$, and rewriting the term $\uh^S \pt \nch\vA$ on the right
side of the result as $\pt(\uh^S(\nch\vA))-(\pt\uh^S)\nch\vA$
yields
\begin{equation}
     \label{eq:4slowuh}
 (1+\epM r^S)(\pt+(\uh^S\cnh))\uh^S -\bh^S\cnh\bh^S   =\pt\,\Xi_2+\xi_2+\nabh\theta_2 \,, \qquad\nch\uh^S=0,
\end{equation}
where
\begin{equation}
  \label{eq:xi1def}
   \Xi_2\eqdef \Xi_1+\uh^S\nch\vA, \qquad \xi_2\eqdef\xi_1-(\pt\uh^S)\nch\vA+\epM\uh^S\az\nch(r\uh^F)
 \end{equation}
 also satisfy
\begin{equation}
  \label{eq:xi2est}
  \|\Xi_2\|_{n-1}+\|\xi_2\|_{n-2}\le c\,\epM^{1+\nu} 
\end{equation}
in view of the estimate \eqref{eq:xi1est} and the same bounds used to obtain that estimate.
 
Now move the $\epM r^S$ term to the right side of \eqref{eq:4slowuh} and replace $\pt\Xi_2$ there
by its divergence-free part, which only changes the divergence term, to obtain
\begin{equation}
     \label{eq:slowuh1}
 (\pt+(\uh^S\cnh))\uh^S -\bh^S\cnh\bh^S   =\pt\,\Xi_{\uh^S}+\xi_{\uh^S}+\nabh\theta_{\uh^S} \,, \qquad\nch\uh^S=0
\end{equation}
with
\begin{equation}
  \label{eq:xiuhSdef}
   \Xi_{\uh^S}\eqdef \cPdh\Xi_2, \qquad \xi_{\uh^S}\eqdef\xi_2-\epM r^S (\pt+(\uh^S\cnh))\uh^S.
\end{equation}
Using the estimate \eqref{eq:xi2est}, the fact from \eqref{eq:ptuhsbhs} that $\pt\uh^S=O(1)$,
and the fact that the projection~$\cPdh$ does not increase Sobolev norms yields
\begin{equation}
  \label{eq:xiuhSest}
  \|\Xi_{\uh^S}\|_{n-1}\le c\,\epM^{1+\nu}, \quad \|\xi_{\uh^S}\|_{n-2}\le c\,\epM
\end{equation}
in view of the term that is explicitly $O(\epM)$ in the definition~\eqref{eq:xiuhSdef} of $\xi_{\uh^S}$.

\begin{proof}[Proof of Theorem~\ref{thm:slow1}]
  The functions $(\uhlim,\bhlim)$ satisfy the systems
  \begin{align}
    \label{eq:uhlim}
    (\pt+(\uhlim\cnh))\uhlim-(\bhlim\cnh)\bhlim&=\nabh\Phi,\qquad \nch\uhlim=0,
    \\
    \label{eq:bhlim}
    (\pt+(\uhlim\cnh))\bhlim-(\bhlim\cnh)\uhlim&=0.
  \end{align}

We now apply Theorem~\ref{thm:pert} to the system \eqref{eq:slowuh1}, \eqref{slow:b:with:ave} for $u\eqdef (\uh,\bh)$ and the system \eqref{eq:uhlim},
\eqref{eq:bhlim} for $U\eqdef (\uhlim, \bhlim)$. Assumptions \eqref{eq:initfis} and \eqref{eq:initfixed}  ensure that the difference in their initial data is $O(\epM^{1+\nu})$.
Since~$\Xi_{\uh^S}$ contains $\cPdh$,   $\Xi_u\eqdef(\Xi_{\uh^S},\Xi_{\bh^S})$ and $\Xi_U\eqdef0$ satisfy $L\Xi_u=0=L\Xi_U$,
where $L=(\begin{smallmatrix}\nch\,&0\end{smallmatrix})$. Define $\xi_u=(\xi_{\uh^S},\xi_{\bh^S})$ and $\xi_U=0$.
The estimates \eqref{eq:xiuhSest}, \eqref{eq:xibhsest} together with the above-mentioned estimate on the difference of the initial data then imply
that the hypotheses of Theorem~\ref{thm:pert} hold with $k=n$, $r=1$, and $\delta=\epM$. Hence the conclusion of that
theorem yields~\eqref{eq:slowest1}.
\end{proof}

\subsection{Equations and estimates for remaining slow modes}\label{s:otherslow}

The third component of equation \eqref{MHDsy:vu} implies that there are no large terms in the PDE for $u_3^S\eqdef \az u_3$, i.e.,
$\left(0,\left(\begin{smallmatrix}0_2\\u_3^S\end{smallmatrix}\right),0_3\right)$ is a zero
eigenvector of the full large operator $\cLA+\mu\cLM$. However, as shown in Lemma~\ref{lem:bdiv0}, the $\mu$-dependent zero eigenvector of
$\cLA+\mu\cLM$ having a
nontrivial projection onto the density component is not just the slow mode $r^S$. Specifically,
\eqref{eq:lalmz} implies that $\vV_r=\left( \az,0_3,\left(\begin{smallmatrix} 0_2\\-\mu\az\end{smallmatrix}\right)\right) $ satisfies
$(\cLA+\mu\cLM)\vV_r\cdot \widetilde\vV=0$ for all functions $\widetilde\vV$, which by the MHD system \eqref{MHDsy} and the
skew-adjointness of $(\cLA+\mu\cLM)$ implies that the PDE satisfied by $\az (r-\mu b_3)=\vV_r\cdot \vV$ will contain no large terms. We therefore need to
calculate the PDE system satisfied by $\az (r-\mu b_3)$ and $\az u_3$. It turns out that while the PDEs for those two variables are coupled by terms
that are strictly $O(1)$, their coupling to other components of the solution contains only terms that are $o(1)$ and so can be considered as small perturbations.

\begin{thm}
  \label{thm:slowru3}
Under the conditions of Theorem~\ref{thm:slow1},
  \begin{equation}
  \label{eq:slowru3est2}
  \|r^S-\rlim\|_{n-2}+\|u_3^S-\uvlim\|_{n-2}\le c \left[\epM^{1-\max(n-5,0)\nu}+|\mu-\mu_{\lim}|\right].
\end{equation}
 \end{thm}

\begin{proof}
Writing the variables in \eqref{eq:azmub3} in terms of fast, intermediate, and slow components  and using the facts that the slow components are
independent of $z$,  $\uh^S$ and $\bh^S$ have zero horizontal divergence,   the vertical averages of $r^I$, $u_3^I$ and $\bh^F$
vanish, and $b_3^S$ is constant in time as well as space transforms that equation into
\begin{equation}
  \label{eq:azmub3sf}
  \begin{aligned}  \pt &(r^S-\mu\az b_3^F)+(\uh^S\cnh) (r^S-\mu\az b_3^F)+\mu(\bh^S\cnh)u_3^S
\\&=-\az\{\nch[(r-\mu b_3)\uh^F]\}-\mu\az[\nch(u_3^I\bh^F)]
\\&=-\nch[(r^S-\mu\az b_3)\az\uh^F]-\az\nch\Big[\big(r^I-\mu (1-\az) b_3^F\big)\uh^F\Big]
{-\mu\az[\nch(u_3^I\bh^F)]}
\end{aligned}
\end{equation}
Replacing $\uh^S$, $\bh^S$, and $\mu$ on the left side of \eqref{eq:azmub3sf} by their limit values, and compensating by adding terms to the right
side yields
\begin{equation}
  \label{eq:rsmazb3}
  \left[ \pt + (\uhlim\cnh)\right](r^S-\mu\az b_3^F)+\mu_{\lim}(\bhlim\cnh)u_3^S=\pt    \Xi_{r^S-\mu\az b_3^F}+\xi_{r^S-\mu\az b_3^F},
\end{equation}
where
\begin{align}\notag
&\Xi_{r^S-\mu\az b_3^F}\eqdef -\nch\left[(r^{S}-\mu{\az} b_3)\int_0^t{\az\uh^F}\right] = -\nch\left[(r^{S}-\mu{\az} b_3)\vA\right]
  \\&\begin{aligned}
    \xi_{r^S-\mu\az b_3^F}&\eqdef \nch\left[(\pt(r^{ S}-\mu{\az} b_3))\vA\right]{-\az\nch\Big[\big(r^I-\mu (1-\az) b_3^F\big)\uh^F\Big]}
    -\mu\az[\nch(u_3^I\bh^F)]
    \\&\qquad+(\uhlim-\uh^S)\cnh(r^S-\mu\az b_3^F)
    +\mu ((\bhlim-\bh^S)\cnh)u_3^S
    +(\mu_{\lim}-\mu)(\bhlim\cnh)u_3^S.
  \end{aligned}
  \label{eq:xiretcdef}
\end{align}
Since $\pt b_3^S\equiv0$, \eqref{eq:azmub3sf} implies a uniform $H^{n-1}$ bound for $\pt(r^S-\mu\az b_3)$. Using in addition the uniform estimate
\eqref{eq:uniform}, estimate \eqref{eq:bigaest} for $\vA$, the estimate \eqref{eq:slowest1} for the convergence rate of the horizontal
components, and Corollary~\ref{cor:azproduct} shows that 
\begin{equation}
  \label{eq:Xiretcest}
  \|\Xi_{r^S-\mu\az b_3^F}\|_{n-1}\le c\,\epM^{1+\nu}, \qquad  \|\xi_{r^S-\mu\az b_3^F}\|_{n-2}\le c\left[ \epM+|\mu-\mu_{\lim}|\right].
\end{equation}

The third component of \eqref{eq:ucons} can be written as
\begin{equation*}
  \pt(\rho u_3)+ \vu\cn(\rho u_3)+(\nc\vu)(\rho u_3) +  \pz {\Phi}-\bh\cnh b_3=0.
\end{equation*}
Applying the vertical averaging operator~$\az$ and
using \eqref{eq:divf} reduces this to
\begin{equation*}
  \pt(\az(\rho u_3))+ \az[\uh\cnh(\rho u_3)]+\az[(\nch\uh)(\rho u_3)] -\az[\bh\cnh b_3^F]=0.
\end{equation*}
In order to treat the term $\az[\bh\cnh b_3^F]$ we write 
$\az b_3^F$ as
\begin{equation}\label{eq:azbrform}
  \begin{aligned}
    \az b_3^F&=-\mu \Deltainv \Delta_h r^S+ \az (b_3^F+\mu\Deltainv\Delta_h r^S)
    =-\mu(r^S-\avv r^S)+ \az (b_3^F+\mu\Deltainv\Delta_h r^S)
    \end{aligned}
\end{equation}
in accordance with the expression estimated in \eqref{eq:fastest}. 
Using \eqref{eq:azbrform} while noting that $(\bh^S\cnh)(\avv r^S)=0$, and using the facts from \eqref{eq:rhodef2}, 
\eqref{eq:ind:z}, and \eqref{eq:modeforms}  that $\rho=1+\epM r$, 
$r^S$ and $u_3^S$ are  independent of
$z$ and  $\az u_3^I=0= u_3^F$, which imply that $\az(\rho u_3^S)=(1+\epM r^S) u_3^S$, $\az(\rho u_3^I)=\epM\az(r^I u_3^I)$, and $\az r=r^S$,
we obtain
\begin{equation}
  \label{eq:4u3:slow}
  \begin{aligned}
  (1+\epM r^S) &\left[\pt +(\uh^S\cnh)\right] u_3^S +u_3^S \left[\pt+(\uh^S\cnh)\right](\epM r^S)+  \mu(\bh^S\cnh) r^S
  \\&=(\bh^S\cnh)[\az (b_3^F+\mu\Deltainv\Delta_h r^S)] +\az\{(\bh^F\cnh)(1-\az)b_3^F\}
  -\epM(\uh^S\cnh)(\az( r^Iu_3^I))
  \\&\quad
  -\epM\pt \az(r^I u_3^I)
    -\nch\big[(\pt\vA)u_3^S\big]-\az\{\nch( \uh^F u_3^I)\}-\epM\az\{\nch(\uh^F ru_3)\},
\end{aligned}
\end{equation}
where the last line results from separating the various modes in $\az\{\nch(\uh^F \rho u_3)\}$ and using the definition of $\vA$ from \eqref{eq:Adef}.
Since there are no terms of size $\epA^{-1}$ in the equations for the time evolution of $r$ or $u_3$, \eqref{eq:4u3:slow} implies that
\begin{equation}
  \label{eq:estptu3s}
  \|\pt u_3^S\|_{n-1}\le c.
\end{equation}

Subtracting $u_3^S$ times \eqref{eq:ptepMrS} from \eqref{eq:4u3:slow},
moving the term $\epM r^S \left[\pt +(\uh^S\cnh)\right] u_3^S$ to the right side of the result, 
noting that the two terms involving $\vA$ partially cancel
and rewriting the remaining term $ (\pt \vA\cnh) u_3^S $ 
as $\pt [ (\vA\cnh) u_3^S]-(\vA\cnh)\pt u_3^S $
yields
\begin{equation}
  \label{eq:u3:slow}
  \begin{aligned}
&    ( \pt+ (\uh^S\cnh))u_3^S+  \mu(\bh^S\cnh) r^S
{= \pt\widetilde\Xi_{u_3^S}+\widetilde\xi_{u_3^S}+(\bh^S\cnh) [\az (b_3^F+\mu\Deltainv\Delta_h r^S)]}
\end{aligned}
\end{equation}
where
\begin{align*}
  &\widetilde\Xi_{u_3^S}= (\vA\cnh) u_3^S -\epM \az(r^I u_3^I),  
  \\
    &\begin{aligned}
    \widetilde\xi_{u_3^S} &=
        \az\{(\bh^F\cnh)(1-\az)b_3^F\}
        -\epM(\uh^S\cnh)(\az(r^I u_3^I))
        -\az\{\nch( \uh^F u_3^I)\}
        -\epM\az\{\nch(\uh^F ru_3)\}
        \\&\quad
        +\epM u_3^S\az\nch(r\uh^F)-(\vA\cnh)(\pt u_3^S)
        -\epM r^S \left[\pt +(\uh^S\cnh)\right] u_3^S . 
\end{aligned}
\end{align*}

As a step towards symmetrizing the system consisting of \eqref{eq:rsmazb3}, \eqref{eq:u3:slow}, we want to replace $r^S$ in the last term on the left side of
\eqref{eq:u3:slow} by $r^S-\mu\az b_3^F$, which requires adding a balancing term involving $\az b_3^F$, which must also be rewritten using
\eqref{eq:azbrform}. This leads us to write $\mu (\bh^S\cnh)r^S$ as $k_1 (\bh^S\cnh)(r^S-\mu\az b_3^F)+k_2 (\bh^S\cnh)(\az b_3^F+\mu r^S)$. Equating
those two expressions and comparing the coefficients of $(\bh^S\cnh)(\az b_3^F)$ shows that $k_2=k_1\mu$, and then comparing the coefficients of
$(\bh\cnh)r^S$ yields $k_1=\frac{\mu}{1+\mu^2}$, 
\begin{equation}
  \label{eq:mubhcnhrs}
  \begin{aligned}
  {\mu (\bh^S\cnh)r^S }
  &=\tfrac{\mu}{1+\mu^2}(\bh^S\cnh)(r^S-\mu\az b_3^F)+\tfrac{\mu^2}{1+\mu^2}(\bh^S\cnh)(\az b_3^F+\mu r^S)
  \\&=\tfrac{\mu}{1+\mu^2}(\bh^S\cnh)(r^S-\mu\az b_3^F)+\tfrac{\mu^2}{1+\mu^2}(\bh^S\cnh)[\az (b_3^F+\mu\Deltainv\Delta_h r^S)],
\end{aligned}
\end{equation}
where the second equation follows as in \eqref{eq:azbrform}.
Substituting \eqref{eq:mubhcnhrs} into \eqref{eq:u3:slow}, moving the second term from \eqref{eq:mubhcnhrs}
to the right side of the resulting equation and combining it with
the similar term already present,
and replacing $\uh^S$, $\bh^S$ and $\mu$ on the left side of the result by their
limiting values and compensating on the right side
yields
\begin{equation}
  \label{eq:u3S}
    \left[\pt+ (\uhlim\cnh)\right] u_3^S
+  \tfrac{\mu_{\lim}}{1+\mu_{\lim}^2}(\bhlim\cnh) (r^S-\mu\az b_3^F)
=\pt\Xi_{u_3^S}+\xi_{u_3^S}
\end{equation}
where
\begin{align}
  \label{eq:hatXiu3S}
  &\Xi_{u_3^S}=  \widetilde\Xi_{u_3^S}+\tfrac1{1+\mu^2} (\bh^S\cnh)\Big [\int_0^t\az (b_3^F+\mu\Deltainv\Delta_h r^S)\Big]
\\
  \label{eq:hatxiu3S}
  &\begin{aligned}
    \xi_{u_3^S}&=\widetilde\xi_{u_3^S}+
    \tfrac1{1+\mu^2} (\pt\bh^S\cnh)\Big [-\int_0^t\az (b_3^F+\mu\Deltainv\Delta_h r^S)\Big]
    + (\uhlim-\uh^S)\cnh u_3^S
        \\&\quad
    +\tfrac{\mu}{1+\mu^2} ((\bhlim-\bh^S)\cnh(\az b_3^F+\mu r^S)
    +(\tfrac{\mu_{\lim}}{1+\mu_{\lim}^2}-\tfrac{\mu}{1+\mu^2})(\bhlim\cnh)(\az b_3^F+\mu r^S).
\end{aligned}
\end{align}
The system consisting of \eqref{eq:rsmazb3},   \eqref{eq:u3S}  can be symmetrized by multiplying the latter equation by~$1+\mu_{\lim}^2$.
The estimates used to obtain \eqref{eq:Xiretcest} together with the time-derivative estimates \eqref{eq:ptuhsbhs}, \eqref{eq:estptu3s} and the
time-integrated estimate~\eqref{TA:epA2} show that 
\begin{align}
  \label{eq:estXiu3S}
  \|\Xi_{u_3^S}\|_{n-1}\le c\,\epM^{1+\nu},
  \qquad
      \|\xi_{u_3^S}\|_{n-2}\le c \left[\epM+|\mu-\mu_{\lim}|\right].
\end{align}

Using \eqref{eq:limit:b3} and the fact that $\avv r^{0,S}$ is a constant, 
the limit equations \eqref{eq:limit:r} and  \eqref{eq:limit:u3} can be rewritten as the system
\begin{align}
  \label{eq:limit:r:mu0:mod}
  &  \left[\partial_t +(\uhlim\cnh)\right] \{(1+\mu_{\lim}^2)\rlim -\mu_{\lim}^2\avv r^{0,S}\} +\mu_{\lim}(\bhlim\cnh) \uvlim=0,
  \\
\label{eq:limit:u3:mu0:mod}
  &\left[ \partial_t+(\uhlim\cnh)\right]\uvlim+\tfrac{\mu_{\lim}}{1+\mu_{\lim}^2}(\bhlim\cnh)\{(1+\mu_{\lim}^2)\rlim-\mu_{\lim}^2\avv r^{0,S}\}=0,
\end{align}
for the dependent variables $(1+\mu_{\lim}^2)\rlim -\mu_{\lim}^2\avv r^{0,S}$ and $\uvlim$,
which has the same form as the system \eqref{eq:rsmazb3},   \eqref{eq:u3S} for the dependent variables
$r^S-\mu\az b_3$ and $u_3^S$,
except that the terms on the right sides are omitted. 
Since the evolution equation for $r$ shows that $\avv r=\avv r^0=\avv r^{0,S}+\epM \avv r^{0,I}=\avv r^{0,S}$,
\begin{equation}
  \label{eq:rfix}
  \begin{aligned}
  r^S-\mu\az b_3^F&=(1+\mu^2)r^S-\mu^2 \avv r-\mu\az (b_3^F+\mu(r-\avv r))
\\&= \left[(1+\mu^2)r^S-\mu^2 \avv r^{0,S}\right]-\mu\az (b_3^F+\mu \Deltainv\Delta_h r).
\end{aligned}
\end{equation}
Hence, by \eqref{eq:initfis},  \eqref{eq:initfixed}, the difference between the initial data for the two systems is
bounded in $H^n$ by a constant times $\epM^{1+2\nu}+|\mu-\mu_{\lim}|$. In view of that bound plus the estimates \eqref{eq:Xiretcest},
\eqref{eq:estXiu3S} for the right sides of \eqref{eq:rsmazb3}, \eqref{eq:u3S}, Theorem~\ref{thm:pert} shows that
\begin{equation}
  \label{eq:slowru3est}
  \begin{aligned}
  \|(r^S-\mu\az b_3^F)-
\{(1+\mu_{\lim}^2)\rlim -\mu_{\lim}^2\avv r^{0,S}\} \|_{n-2}&+
\|u_3^S-\uvlim\|_{n-2}
\le c \left[\epM +|\mu-\mu_{\lim}|\right].
\end{aligned}
\end{equation}
By \eqref{eq:rfix}, the static estimate \eqref{eq:fastest} with $j=n-3$ applied 
to the $-\mu\az (b_3^F+\mu \Deltainv\Delta_h r)$ term of \eqref{eq:rfix}
shows that \eqref{eq:slowru3est} implies that \eqref{eq:slowru3est2} holds.
\end{proof}

  As discussed in the introduction, the term $|\mu-\mu_{\lim}|$ is the dominating error term in \eqref{eq:slowru3est} and \eqref{eq:slowru3est2} whenever $\mu_{\lim}=0$, but
  that term will be eliminated in Theorem~\ref{thm:ru3cor} below by adding corrector terms.

\begin{thm}\label{thm:ru3cor}
Let $(\rcor,\uvcor)$ be the solution of the inhomogeneous linear system
      \begin{align}
        \label{eq:rcor}
        &\begin{aligned}
          \pt \rcor+(\uhlim\cnh)\rcor& +\mu(\bhlim\cnh)\uvcor
          =-(\bhlim\cnh)\uvlim-(\mu+\mu_{\lim})(\pt+(\uhlim\cnh))\rlim,
        \end{aligned}
        \\
        &\begin{aligned}\pt\uvcor&+(\uhlim\cnh)\uvcor+\tfrac{\mu}{1+\mu^2}(\bhlim\cnh)\rcor
          =
          -\tfrac{1-\mu\mu_{\lim}}{(1+\mu^2)(1+\mu_{\lim}^2)}(\bhlim\cnh)((1+\mu^2)\rlim-\mu^2\avv r^{0,S})
        \end{aligned}
        \label{eq:uvcor}
      \end{align}
having initial data zero. If the conditions of Theorem~\ref{thm:slow1} hold then
\begin{equation}
  \begin{aligned}
    \|r^S-(\rlim+\tfrac{\mu-\mu_{\lim}}{1+\mu^2}\rcor)\|_{n-2}+\|u_3^S-(\uvlim+(\mu&-\mu_{\lim})\uvcor)\|_{n-2}
    \le c\, \epM^{1-\max(n-5,0)\nu}.
\end{aligned}
\label{eq:slowest2b}
\end{equation}
\end{thm}

\begin{proof}
  Since $(\mu-\mu_{\lim})(\mu+\mu_{\lim})=\mu^2-\mu_{\lim}^2$ and $\avv r^{0,S}$ is a constant,
  adding $\mu-\mu_{\lim}$ times \eqref{eq:rcor} to \eqref{eq:limit:r:mu0:mod} yields
  \begin{equation}
    \label{eq:rpluscor}
    \begin{aligned}&\left[\partial_t +(\uhlim\cnh)\right] \{(1+\mu^2)\rlim -\mu^2\avv r^{0,S}+(\mu-\mu_{\lim})\rcor\}
      +\mu(\bhlim\cnh) (\uvlim+(\mu-\mu_{\lim})\uvcor)=0,
  \end{aligned}
\end{equation}
Similarly, since $(\mu-\mu_{\lim})$ times $\tfrac{1-\mu\mu_{\lim}}{(1+\mu^2)(1+\mu_{\lim}^2)}$ equals $\frac{\mu}{1+\mu^2}-\frac{\mu_{\lim}}{1+\mu_{\lim}^2}$, adding
$\mu-\mu_{\lim}$ times \eqref{eq:uvcor} to \eqref{eq:limit:u3:mu0:mod} yields
 \begin{equation}
    \label{eq:u3pluscor}
    \begin{aligned}&\left[\partial_t +(\uhlim\cnh)\right] \{\uvlim+(\mu-\mu_{\lim})\uvcor\}
      +\tfrac{\mu}{1+\mu^2}(\bhlim\cnh) \{(1+\mu^2)\rlim -\mu^2\avv r^{0,S}+(\mu-\mu_{\lim})\rcor\}=0.
  \end{aligned}
\end{equation}
Equations \eqref{eq:rpluscor}--\eqref{eq:u3pluscor} have the same form as as the system \eqref{eq:rsmazb3}, \eqref{eq:u3S} for the dependent variables
$r^S-\mu\az b_3$ and $u_3^S$, except that the terms on the right sides are omitted and all occurrences of $\mu_{\lim}$ in the coefficients
on the left sides are replaced
by $\mu$. Omitting the step of replacing  $\mu$ by $\mu_{\lim}$ in  the derivation of \eqref{eq:rsmazb3}, \eqref{eq:u3S}  yields those equations with 
all occurrences of $\mu_{\lim}$ on the left sides replaced by $\mu$ and the terms of order $\mu-\mu_{\lim}$ omitted from their right sides.
Since the terms of order $\mu^2$ in \eqref{eq:rpluscor} now involve $\mu^2$ rather than $\mu_{\lim}^2$, as in \eqref{eq:rfix} and in contrast to \eqref{eq:limit:r:mu0:mod},
there is also no longer a term of size $O(|\mu-\mu_{\lim}|)$ in the difference in the initial data. Hence
applying Theorem~\ref{thm:pert} now yields an estimate without the term involving $|\mu-\mu_{\lim}|$, and by using \eqref{eq:rfix} the estimate
so obtained can be written as \eqref{eq:slowest2b}. 
\end{proof}

\appendix
\section{Derivation of the MHD system}\label{sec:curl}

Suitably scaled, the motion of an isentropic compressible, conducting, inviscid fluid is
modeled by the MHD system (\cite[\S 3.8]{Davidson})
 \begin{subequations}\label{MHD:0}\begin{align}
 \label{rho:0}\pt\rho+\nc(\rho\vu)&=0\\
 \label{vu:0}\pt(\rho\vu)+\vu\cn (\rho\vu)+(\nc\vu)\rho\,\vu+\epM^{-2}
 \grad p(\rho)+\epA^{-2}{{ \vB\times( {\nabla\!\times\!}\vB)  } } &=0,\\
\label{vB:0}\pt\vB-{\nabla\!\times\!}(\vu\times\vB)&=0,\\
 {\nc\vB}&=0.
\end{align}
\end{subequations}
Here $\epM$ denotes the well-known Mach number, $\rho$ is the fluid density, $p(\rho)$ is the pressure
law that satisfies $p'>0$,  $\vu$ is the fluid velocity, and $\vB$ is the magnetic field.
The parameter $\epA$, as we call
the \Al number in this article, is the ratio between flow velocity and speed
of the magnetosonic waves; in \cite{KM81} the \Al number is the
reciprocal of our version.

We consider the case in which a uniform magnetic field is applied in the
direction~$\vez$ parallel to the $z$-axis, which subjects the fluid to a large
Lorenz force. To reformulate the system \eqref{MHD:0} into a form to which the
results of \cite{CJS17} can be applied,  we begin by rescaling the magnetic field and the density via
\begin{equation}\label{eq:rhodef}
 \vB=\vez+\epA\vb,\qquad \rho=1+\epM r.
\end{equation}
Applying calculus identities for the curl,
subtracting $\vu$ times \eqref{rho:0} from
\eqref{vu:0}, and multiplying~\eqref{rho:0} by $a(\epM r)$ from \eqref{eq:rhodef2}
yields the system \eqref{MHDsy}.


\section{Improved Uniform Bound}\label{sec:bound}

\begin{lem}\label{lem:bound}
      Let $n\ge s_0+1$ be an integer, where $s_0\eqdef\lfloor \frac d2\rfloor+1$ is the Sobolev
embedding exponent, i.e., the smallest integer $s$ for which
$\|f\|_{L^\infty}\le c\|f\|_{H^s}$.
Assume that the spatial domain is $\mR^d$ or $\mT^d$ and that the system \eqref{eq:general} and its initial data $\vV^0$ satisfy the following conditions:
\begin{enumerate}
\item the operators $\cLA$ and $\cLM$ are  constant-coefficient
  differential operators of order at most one and are skew-adjoint on $L^2$,

\item the matrices $A_i$ are smooth symmetric functions for $j\ge0$ and the
  matrix $A_0$ is positive definite,

\item  the small parameters are restricted to the region \eqref{eq:epAepMgen},

\item the initial data $\vV^0$, which may depend on the small parameters $\epA$ and $\epM$, are  uniformly bounded
  in $H^n$ and satisfy  the ``well-preparedness'' condition \eqref{eq:wellprep}.
\end{enumerate}

 Then there
  exist fixed positive  $T$ and $K$ such that for $(\epA,\epM)$ satisfying \eqref{eq:epAepMgen}
  the solution to \eqref{eq:general}
  having
  the initial data $\vV^0$ exists for $0\le t\le T$ and satisfies  \eqref{eq:uniform};
in particular the solution is uniformly bounded in $H^n$.

\end{lem}

\begin{proof}
  Lemma~\ref{lem:bound}
  differs from Theorem~3.6 of \cite{CJS17}
  only by having different weights multiplying the norms of time derivatives. Hence it suffices
  to show that in all places in the proof of \cite[Theorem 3.6]{CJS17} where
  the use of the weights
\begin{equation}
  \label{eq:oldest}
  \max_{0\le t\le T} \left[ \sum_{j=0}^n\epM^{j}\|\pt^j \vV\|_{{n-j}}+\|\vV_t\|_{0}  \right]\le K
\end{equation}
  was justified the use of the improved
  weights in \eqref{eq:uniform} is also justified.
  There are only two such places,
    namely where it was shown that the weighted sum of norms is bounded
  at time zero and where it was shown that the small parameters scale out of
  the estimate for the time derivative of an appropriately  weighted
  energy.

  As noted in \cite[proof of Lemma 3.5]{CJS17}, assumption
  \eqref{eq:wellprep} ensures that $\|\vV_t\eval{t=0}\|_{{n-1}}$ is bounded uniformly
  in the small parameters and the PDE~\eqref{eq:general} then yields the estimates
  $\|\pt^j\vV\eval{t=0}\|_{{n-j}}\le c\,\epA^{1-j}$ for $1\le j\le n$. Therefore, for $1\le
  j\le n$,
  \begin{equation*}
    \begin{aligned}
    \epM^{j-1}&\left(\min\left(\tfrac{\epA}{\epM},1\right)\right)^{n-1}\|\pt^j\vV\eval{t=0}\|_{{n-j}}\le
    \epM^{j-1}\left(\min\left(\tfrac{\epA}{\epM},1\right)\right)^{n-1}
    \left( c\,\epA^{1-j}\right)
    \le  c \left(\min\left(\tfrac{\epA}{\epM},1\right)\right)^{n-j}\le c,
  \end{aligned}
\end{equation*}
which shows that the weighted sum of norms in \eqref{eq:uniform} is also
bounded at time zero uniformly in the small parameters.

The energy estimate both in \cite{CJS17} and here makes use of the norms
\begin{equation}
  \label{eq:normdef}
  \|f\|_{\ell,A_0}\eqdef \sqrt{\sum_{0\le|\alpha|\le \ell}\int
    (D^\alpha f)^TA_0(\epM \vV)D^\alpha f\,dx},
\end{equation}
where $\vV$ is a solution to \eqref{eq:general} and
$D^\alpha=\partial_{x_1}^{\alpha_1}\cdots\partial_{x_d}^{\alpha_d}$.
As shown in \cite{CJS17}, in
order to prove a weighted energy estimate like \eqref{eq:oldest} or
\eqref{eq:uniform} it suffices to obtain a uniform bound for
\begin{equation}\label{eq:Edef}
  E\eqdef\|\vV\|_{n,A_0}^2+\|\vV_t\|_{0,A_0}^2+\sum_{j=1}^n
  w_j^2\|\pt^j\vV\|_{n-j,A_0}^2,
\end{equation}
where the weights $w_j$ are $\epM^{j}$ for the
estimate \eqref{eq:oldest} or
\begin{equation}\label{eq:weightsdef}
  w_j=\epM^{j-1}\left(\min\left(\tfrac{\epA}{\epM},1\right)\right)^{n-1},\quad 1\le j\le n
\end{equation}
for the estimate~\eqref{eq:uniform}.  
Moreover, in the estimates \cite[(3.12),  (3.24)]{CJS17} for $\ddt E$, the
only facts used about the weights $w_j$ to prove a uniform bound for $E$ are
that for some finite $c$ that may be different in each appearance
\begin{subequations}\label{eq:weightconds}
\begin{align}
  \label{eq:w1cond}
  \epM&\le c\, w_1,
  \\
  \label{eq:wjj+1}
  \epM w_j&\le c\, w_{j+1} &&\text{for $1\le j\le n-1$,}
  \\
  \label{eq:weightbreak}
  w_k&\le c \prod_{j=1}^J w_{k_j}&&\text{whenever $\sum_{j=1}^J k_j=k$,}
\\
  \label{eq:weightcombine}
  \epM w_k&\le c \prod_{j=1}^J w_{k_j}&&\text{whenever $\sum_{j=1}^J k_j=k+1$.}
\end{align}
\end{subequations}
Since \eqref{eq:weightcombine} can be obtained by substituting
\eqref{eq:weightbreak} with $k$ replaced by $k+1$ into \eqref{eq:wjj+1} with
$j$ set equal to $k$, it
suffices to prove \eqref{eq:w1cond}--\eqref{eq:weightbreak}. The
definitions~\eqref{eq:weightsdef} imply that \eqref{eq:wjj+1} holds provided
that $c$ there is at least one, while both \eqref{eq:w1cond} and \eqref{eq:weightbreak} reduce to
the condition $\epM\le c \left(\min\left(\frac{\epA}{\epM},1\right)\right)^{n-1}$ that is
equivalent to \eqref{eq:epAepMgen}. 
\end{proof}

Combining estimate \eqref{eq:uniform} with the standard Sobolev interpolation inequality
$\|f\|_{r}\le C_{r,s}\|f\|_s^{\frac rs}\|f\|_0^{1-\frac rs}$  for $0\le r\le s$
   (e.g. \cite[(2.32)]{Ma84}), yields the following result.
\begin{cor}\label{cor:unifinterp}
  When the basic conditions of Theorem~\ref{thm:theom1} hold and $\mu\le1$ then
  \begin{equation}
    \label{eq:vtinterp}
    \|\vV_t\|_j\le c \mu^{-j}=c\,\epM^{-j\nu} \qquad j=0,\ldots,n-1.
  \end{equation}
\end{cor}

\section{Convergence and Limit}\label{sec:conv}

The convergence part of Theorem~\ref{thm:theom1} follows from \cite[Theorem 4.6]{CJS17} when $\mu_{\lim}=0$, and from simple modifications of convergence
results for two-scale singular limits when $\mu_{\lim}>0$. Since we need explicit formulas for the limit equations and will use some of the formulas
derived below in \S\ref{sec:rate} we indicate a direct unified proof.

\begin{proof}[Proof of the convergence part of Theorem~\ref{thm:theom1}]

  The uniform bounds on $\vV$ and $\vV_t$ provide compactness, which together with the uniqueness of solutions to the limit equations ensures
  the convergence of $\vV$,  in $C^0([0,T];H^{n-\alpha})$ for any $\alpha>0$ and weak-$*$ in $L^\infty([0,T];H^n)$,
  to a limit $\vVlim$ as $\epA,\epM$ tend to zero with their ratio converging to a given limit $\mu_{\lim}$, with $\vV_t$
 converging weak-$*$ in $L^\infty([0,T],H^{n-1})$ to $\vVlim_t$.

 Multiplying \eqref{eq:general} by $\epA$  or applying $\epM\mathbb{P}^0$ to it, and taking the limit yields
\begin{equation}\label{eq:PgenVlim0}
   (\cLA+\mu_{\lim}\cLM)\vVlim=0=\mathbb{P}^0\cLM\vVlim. 
 \end{equation}
Identities \eqref{eq:PgenVlim0} and Lemma~\ref{lem:bdiv0} imply that
$\vVlim$ is independent of $z$, the horizontal parts of its velocity and magnetic field are divergence free, and
\eqref{eq:limit:b3} holds.
If the spatial domain is $\mR^3$ then $\vVlim$ must therefore vanish, so from now on that domain is $\mT^3$.
By \eqref{def:P0} and \eqref{eq:lalmv}, $\mathbb{P}\left(\fr{\epA}\cLA+\fr{\epM}\cLM\right) =(\tfrac1{\epM}\az\nch\uh, 0_3, 0_3)$. 
Taking the limit of the equations with no large terms in $\mathbb{P}$ applied to \eqref{MHDsy} yields \eqref{eq:limit:uh}, \eqref{eq:limit:u3}, 
\eqref{eq:limit:bh}. 

To determine the limit equation for the density, divide \eqref{MHDsy:r} by $a(\epM r)$, which puts it in conservation form, and apply $\az$ to obtain
\begin{equation}
  \label{eq:rforlim}
  \pt(\az r) +\az[\nch(r \uh) ]+\tfrac1\epM\az(\nch\uh)=0.
\end{equation}
To eliminate the large term in \eqref{eq:rforlim}, write the third component of \eqref{MHDsy:vb} in conservation form as
$\pt b_3+\nc(b_3\vu)-\nc(u_3\vb)+\epA^{-1}\nch\uh=0$, apply $\mu\az$ and subtract the result from 
\eqref{eq:rforlim}, which yields
\begin{equation}
  \label{eq:azmub3}
  \pt[\az( r-\mu b_3)]+\az\{\nch[(r-\mu b_3)\uh]\}+\mu\az[\nch(u_3\bh)]=0.
\end{equation}
Taking the limit of \eqref{eq:azmub3}, using facts that $\vVlim$ is independent of $z$ and that $\uhlim$ and $\bhlim$ are divergence free,
and substituting \eqref{eq:limit:b3} into the result yields \eqref{eq:limit:r}.
\end{proof}

\section{Perturbation Theorem}\label{sec:pert}

The following perturbation theorem is a variant of \cite[Lemma 3.2]{Cheng:Mach}), and can be proven by similar methods.
\begin{thm}\label{thm:pert}
    Suppose that $u$ and $U$ are solutions in $C^0([0,T];H^k)$ of
  \begin{align}
    \label{eq:witherr}
    A_0(u)u_t+\sum_{i=1}^d A_i(u)u_{x_i}&=F+L^*v+A_0(u)\pt \Xi_u+\xi_u, &\qquad &Lu=0,
    \\
    \label{eq:noerr}
    A_0(U)U_t+\sum_{i=1}^d A_i(U)U_{x_i}&=F+L^*V+A_0(U)\pt \Xi_U+\xi_U,&\qquad &LU=0,
  \end{align}
  having the same initial value $u_0\in H^{k}$, where $k\ge \lfloor \frac d2\rfloor+2$,
  the matrices $A_i$ are smooth and symmetric and $A_0$ is
  positive-definite, $F$ is a given function of $t$ and $x$, $L$ is a first-order differential operator with constant
  coefficients, with $L^*$ denoting its $L^2$-adjoint,
  and $\Xi_u$, $\Xi_U$, $\xi_u$, and $\xi_U$ satisfy
  \begin{align*}
    &\|\Xi_u\|_{{k-r}}+\|\Xi_U\|_{{k-r}}+\|\xi_u\|_{{k-r-1}}+\|\xi_U\|_{{k-r-1}}\le c\delta
    \quad\text{for some $0\le r\le k-1$,}
    \\
    &L\Xi_u=0=L\Xi_U, \qquad \text{and}\qquad
    \|\pt\Xi_U\|_{{k-r-1}}\le c.
  \end{align*}
  Then $\max_{0\le t\le T}\|u-U\|_{{k-r-1}}\le c\delta$.
\end{thm}

\section{Calculus Inequalities for vertical averages} \label{sec:azprod}

The following result is sharper than what would be obtained by the standard product estimate (e.g., \cite[Proposition 2.1A]{Ma84},
   because the entire product is estimated using the
  $W^{1,1}$ norm rather than pulling out one factor in the $L^\infty$ norm, and  the Gagliardo-Nirenberg inequalities are
  used in dimension two rather than three.

\begin{lem}\label{lem:azfgHk}
  For all $j\ge1$ there exists a constant $C_j$ such that for $f,g\in H^j(\mT^3)$
  \begin{equation}
    \label{eq:azfgnew}
    \|\az(fg)\|_{H^{j-1}(\mT^2)}\le C_j\left( \|f\|_{H^{j}(\mT^3)}\|g\|_{L^2(\mT^3)}+\|f\|_{L^2(\mT^3)}\|g\|_{H^j(\mT^3)}\right).
  \end{equation}
\end{lem}
\begin{proof}
We first prove \eqref{eq:azfgnew} for $j=1$: By the Gagliardo-Nirenberg inequality $\|h\|_{L^2(\mT^2)}\le c \|h\|_{W^{1,1}(\mT^2)}$
and  the Cauchy-Schwartz inequality,
\begin{equation*}
  {\|\az(fg)\|_{L^2(\mT^2)}}\le c\|\az(fg)\|_{W^{1,1}(\mT^2)}
    \le c\!\left( \|f\|_{H^1(\mT^3)}\|g\|_{L^2(\mT^3)}+\|f\|_{L^2(\mT^3)}\|g\|_{H^1(\mT^3)}\right).
\end{equation*}
Now let $j$ be any integer greater than one. By  the definition of the $H^{j-1}$ norm, the result for the case $j=1$, the Sobolev interpolation
inequality, and Young's inequality for products $a^\sigma b^{1-\sigma}\le a+b$ for $0\le\sigma\le1$, 
\begin{align}\label{eq:forcasek}
    \|&\az(fg)\|_{H^{j-1}(\mT^2)}\le c\sum_{|\alpha|\le j-1}\|\az(D^\alpha_{x,y}(fg))\|_{L^2(\mT^2)}
  \\\notag&
           \le c\sum_{|\beta|+|\gamma|\le j-1}\|\az((D^\beta_{x,y}f)(D^\gamma_{x,y}g))\|_{L^2(\mT^2)}
  \\\notag&
            \le c \sum_{|\beta|+|\gamma|\le j-1}\left(\|f\|_{H^{|\beta|+1}(\mT^3)}\|g\|_{H^{|\gamma|}(\mT^3)}+\|f\|_{H^{|\beta|}(\mT^3)}\|g\|_{H^{|\gamma|+1}(\mT^3)}\right)
  \\\notag&\le c\sum_{0\le i\le j-1}\left( \left[ \|f\|_{H^{i+1}}\|g\|_{L^2}+\|g\|_{H^{i}}\|\|f\|_{H^1}\right]+
  \left[ \|f\|_{H^i}\|g\|_{H^1}+\|g\|_{H^{i+1}}\|f\|_{L^2} \right]\right)
  \\\notag&\le c \left( \|f\|_{H^j}\|g\|_{L^2}+\|g\|_{H^{j-1}}\|f\|_{H^1}+\|f\|_{H^{j-1}}\|g\|_{H^1}+\|g\|_{H^j}\|f\|_{L^2}\right)
\end{align}
A second application of the Sobolev interpolation inequality followed by Young's inequality shows that each of the terms in the final line of
\eqref{eq:forcasek} in which the $H^{j-1}$ and $H^1$ norms appear is bounded by the sum of the two terms there in which the $H^j$ and $L^2$ norms
appear, which yields \eqref{eq:azfgnew} for $j>1$.
\end{proof}

\begin{cor}\label{cor:azp}
  Consider integer $n\ge3$ and a geometric sequence $\{\ep_j\}$ with common ratio $\frac 1\mu\ge 1$ and $\ep_n\le c$. 
Suppose $v(x,y,z),w(x,y,z)\in {H^n(\mT^3)}$ satisfy the ``interpolative estimates''
\begin{equation*}
\big\|(v,w)\big\|_{H^j(\mT^3)}\le
c\,\ep_j,\quad j=0, \ldots, n-1,\qquad \big\|(v,w)\big\|_{H^n(\mT^3)}\le c.
\end{equation*}
Then
\begin{equation}\label{eq:azvwest}
\|\az(vw)\|_{H^{n-1}(\mT^2_{x,y})} \le c \ep_0,\qquad \|\az(vw)\|_{H^{n-2}(\mT^2_{x,y})} \le c \ep_0\,\mu.
\end{equation}
\end{cor}
\begin{proof}
  By Lemma~\ref{lem:azfgHk}, $\|\az(vw)\|_{j-1}\le c(\|v\|_j\|w\|0+\|v\|_0\|v\|_j)\le c \ep_j\ep_0$. Since $\ep_n\le c$ and $\ep_{n-1}\le \mu\ep_n\le
  c\mu$, this implies \eqref{eq:azvwest}.
\end{proof}

In view of the uniform $H^n$ estimate \eqref{eq:uniform}, the static estimates \eqref{eq:fastest}, \eqref{eq:inteststat}, and the
relations~\eqref{eq:epAepMgen}, \eqref{eq:nudef} between the parameters, Corollary~\ref{cor:azp} yields the following estimates for products of components of various modes.
\begin{cor} \label{cor:azproduct} Assume that the basic conditions of Theorem~\ref{thm:theom1} hold. Let $v^F,w^F$ be either $(1-\az) b_3^F$ or any
  component of $\vV^F$ except $b_3^F$ and let $v^I,w^I$ be any component of $\vV^I$. Then,
  \begin{equation*}
    \begin{aligned}
      \sup_{0\le t\le T}\big\{\|&\az(v^Fw^F)\|_{n-1}+\mu\|\az(v^Iv^F)\|_{n-1}+\mu^2\|\az(v^Iw^I)\|_{n-1}
      +\|\az\nc( v^I\uh^F)\|_{n-2}+\epM\|v^F\|_{n-1}\big\}\le c\,\epA,
    \end{aligned}
  \end{equation*}
and the estimates also hold when $\mu$ or $\mu^2$ on the left side is replaced by $\epM$.
\end{cor}
The  estimate of $\epM\|v^F\|_{n-1}$ does not use Lemma~\ref{lem:azfgHk} but is included for convenience.

\begin{ack}
Ju is supported by the NSFC (Grants No.11571046, 11471028,
      11671225). Cheng and Ju are supported by the UK Royal Society ``International Exchanges'' scheme (Award No. IE150886).  Ju and Schochet are
      supported by the ISF-NSFC joint research program (NSFC Grant No. 11761141008 and ISF Grant No. 2519/17). Cheng is supported by the Leverhulme
      Trust (Award No. RPG-2017-098) and the EPSRC (Grant No. EP/R029628/1).
\end{ack}  


\end{document}